\documentclass[a4paper]{article}

\usepackage{helvet}
\usepackage{fullpage}
\usepackage{courier}
\usepackage[T1]{fontenc}
\usepackage[utf8]{inputenc}

\usepackage{mathtools,amsfonts,amssymb,amsmath,amsthm}
\usepackage{algcompatible}
\usepackage{mathtools}
\usepackage{amssymb}

\usepackage{pgf,tikz,pgfplots}
\usepackage{tikz-3dplot}
\usetikzlibrary{shapes,arrows,shadows,patterns,pgfplots.groupplots,decorations.markings,chains}
\usepackage{pgfplots}
\pgfplotsset{compat=newest}
\usepackage{mathrsfs}

\usepackage{color,ulem}
\usepackage{graphicx}%
\usepackage[active]{srcltx}
\usepackage{mathabx}
\usepackage{verbatim}
\usepackage{cancel}
\usepackage{tikz,pgfplots}
\usepackage{tabu}

\usepackage{hyperref}
\hypersetup{
  linkcolor  = green!60!black,
  citecolor  = blue!60!black,
  urlcolor   = red!60!black,
  colorlinks = true,
}

\usepackage{caption}
\usepackage{subcaption}

\usepackage[official]{eurosym}

\usepackage[boxruled]{algorithm2e} % to typeset algorithms
\SetKwRepeat{Do}{do}{while}% for do-while loops

\newcommand{\dd}{\textup{d}}
\def \R{\mathbb R}
\def \RR{\mathbb R}
\def \N{\mathbb N}

\newcommand\p\varphi
\newcommand\e\epsilon

\definecolor{battleshipgrey}{rgb}{0.52, 0.52, 0.51}
\definecolor{forestgreen}{rgb}{0.0, 0.27, 0.13}
\definecolor{dgreen}{RGB}{0,80,0}
\definecolor{gray}{RGB}{128,128,128}
\definecolor{ffqqqq}{rgb}{1.,0.,0.}
\definecolor{qqqqff}{rgb}{0.,0.,1.}

\definecolor{orange}{RGB}{255,127,0}
\definecolor{purple}{RGB}{255,0,255}
\definecolor{DP}{RGB}{100,0,100}
\definecolor{shadebrushcolor}{gray}{0.9}
\definecolor{hyperlinkcolor}{gray}{0.4}

\usepackage[overload]{empheq}

\usepackage{command-estim}

\begin{document}

%%-----------------------------
%%      the top matter
%%-----------------------------
\title{Mortensen Observer for a class of variational inequalities -- Lost equivalence with stochastic filtering approaches}

\author{L.-P. Chaintron$^{1,2}$, Á. Mateos$^{3}$ González, L. Mertz$^{3}$, P. Moireau$^{2}$ \\
{\small $^{1}$ DMA, \'Ecole Normale Sup\'erieure 45 rue d’Ulm, 75005 Paris, France;}\\
{\small$^{2}$Inria -- LMS, Ecole Polytechnique, CNRS -- Institut Polytechnique de Paris; Palaiseau, France;}  \\
{\small$^{3}$ECNU-NYU Institute of Mathematical Sciences, NYU Shanghai, Shanghai, China; }}
\date{}
\maketitle

%\author{Auth2}\sameaddress{1}
%\address{Address;\email{email}}

%%
\begin{abstract} 
We address the problem of deterministic sequential estimation for a nonsmooth dynamics in $\RR^+$ governed by a variational inequality, as illustrated by the Skorokhod problem with a reflective boundary condition at $0$. 
For smooth dynamics, Mortensen \cite{Mortensen:1968vy} introduced an energy for the likelihood that the state variable produces --~up to perturbations disturbances~-- a given observation in a finite time interval, while reaching a given target state at the final time. The Mortensen observer is the minimiser of this energy. For dynamics given by a variational inequality and therefore not reversible in time, we study the definition of a Mortensen estimator.
On the one hand, we address this problem by relaxing the boundary constraint of the synthetic variable and then proposing an approximated variant of the Mortensen estimator that uses the resulting nonlinear smooth dynamics.
On the other hand, inspired by the smooth dynamics approach in~\cite{JB88}, we study the vanishing viscosity limit of the Hamilton-Jacobi equation satisfied by the Hopf-Cole transform of the solution of the robust Zakai equation. We prove a stability result that allows us to interpret the limiting solution as the value function associated with a control problem rather than an estimation problem. In contrast to the case of smooth dynamics~\cite{Hijab:1980,JB88,Fleming:1997wm}, here the zero-noise limit of the robust form of the Zakai equation cannot be understood from the Bellman equation of the value function arising in Mortensen's deterministic estimation. This may unveil  a violation of equivalence for non-reversible dynamics between the Mortensen approach and the low noise stochastic approach for nonsmooth dynamics.
%We address the problem of establishing a deterministic sequential estimation strategy for a nonsmooth dynamics in $\RR_+$, illustrated here by the Skorokhod problem with a reflective boundary condition at $0$. On the one hand, this problem can first be tackled with a deterministic filtering approach by relaxing the boundary constraint with a penalizing drift and applying the Mortensen estimator \cite{Mortensen:1968vy} to the resulting non-linear smooth dynamics. On the other hand, following \cite{JB88}, we can relate the deterministic filtering problem to a stochastic filtering problem with a vanishing noise. By analyzing the underlying Hamilton-Jacobi equation with vanishing viscosity and Neumann boundary condition at 0, we derive a limit problem and understand in the nonsmooth case the violation of equivalence between the Mortensen approaches and the small noise stochastic approaches in contrast to the has been classically obtained for reversible systems \cite{Hijab:1980,JB88,Fleming:1997wm}.
\end{abstract}

%We prove a stability result in the vanishing viscosity limit. This equation is associated with the Hopf-Cole transform of the robust Zakai equation for the unnormalized density of the hidden state.
%
%Our approach bears many similarities with that of James and Baras~\cite{JB88}. However, we choose a different order. We first define a function $w^\varepsilon$ as the solution of a viscous second-order evolution Hamilton-Jacobi equation that we posit. This allows us to deal with the simplest Neumann boundary condition we can, and once we have proved well-posedness and stability results, we define our counterpart of James and Baras' cost function $\entropy^\varepsilon$ as a modification of $w^\varepsilon$ and recover the corollary results as corollaries. This alternative strategy dispenses us of having to prove well-posedness results on the related Zakai equations.

\tableofcontents

%
% It is also reasonably consistent with the observation of~\cite[fig.3c]{Bianchi2017} that at a distance from the boundary higher than $a/4$, where $a$ is the micro-swimmer half-length, the measured swim speeds are roughly constant. 

\section{Introduction}

%[paragraph1: we consider/we aim]
In this paper, we consider the problem of estimating the deterministic state resulting from a nonsmooth dynamical system given an observation. The system state is the solution of a variational inequality,  and both the state dynamics and observation are subjected to disturbances and the objective is to find the ``best'' deterministic estimate of the state from the observation.  
%%[paragraph2: motivating]
The problem is motivated by the state estimation of fundamental nonsmooth dynamical systems related to a) elasto-plasticity (transition from elastic and plastic phases) \cite{Duvautea76}, b) dry friction (transition from static and dynamic phases) \cite{Bastienea00} or c) impacts (switch of velocity at the instant of contact with an obstacle) \cite{Bernardinea13}. The state variable in such models is non-differentiable at the transition from a phase to one another. These models can be represented using the framework of variational inequalities. It is worth mentioning that the Skorokhod problem with a reflective boundary condition at 0 belongs to this kind of framework.
%%1 Bastien + Lamarque (2000)
%%2 Duvaut Lions (1968)  
%%3 Livre de F. Bernardin (2012) 
%%[paragraph3: standard methods]
For smooth dynamical systems an ``optimal'' deterministic approach to non-linear system filtering was made by Mortensen \cite{Mortensen:1968vy}. He proposes to minimise an energy associated with the likelihood that the state variable produces --~up to disturbances~-- a given observation on a finite time interval while reaching a given target state at the final time. The lower the energy of a target state, the more likely it is. The Mortensen filter is the minimiser of this energy, also known since then as the minimum energy estimator \cite{Hijab:1980,Krener:2015}. Moreover, Mortensen also proposed a differential equation for the time dependent dynamics of this estimator. Then, several authors, see for instance \cite{Hijab:1980,JB88,Fleming:1997wm} made the connection between Mortensen's approach and fundamental methods of stochastic filtering where both the state and observations disturbances are small amplitude white noises. The central tool of stochastic filtering is the Zakai equation, whose solution is the unnormalised conditional density of the state given the observation \cite{zakai1969optimal,Jazwinsky70,xiong2008}. They obtained the minimum energy approach as the  zero noise limit of the robust form (path-wise form) of the Zakai equation. The proof of  \cite{Hijab:1980} is essentially probabilistic and relies on large deviations theory. In the proof of \cite{JB88}, Mortensen's energy evolves according to a nonlinear PDE of Hamilton-Jacobi-Bellman type, and is interpreted as a solution in the viscosity sense.  
%%4 Mortensen
%%5 hijab
%%6 James et Baras
%%7 Zakai 
%%8 Zakai robuste
%[paragraph4: our contribution]
In this paper we study the extension of the Mortensen estimator for  dynamics. First, we propose an approximation of the Mortensen estimator using a penalisation approach. Then, we extend the analysis of \cite{JB88} to the case of non-smooth dynamics in $\RR_+$. We study the vanishing viscosity limit of the Hamilton-Jacobi equation satisfied by the Hopf-Cole transform of the solution of the robust Zakai equation, and we prove a stability result that allows us to interpret the limiting solution as the value function associated with a control problem, but not as the value function of the minimum energy estimation problem, as was the case for time-reversible dynamics.
%Our present follows �
%%Extension de la m\'ethode PDE de James et Barras dans le cas d'un modele nonsmooth particulier. 
%[REVIEW OF JB HERE]
%%[paragraph5: take away message]
%%Le probl�me limit� ne tombe sur le probl�me d�estimation souhaite.
%%[paragraph6: organisation du papier]
The paper is organized as follows. In Section 2 we introduce the nonsmooth dynamical systems at stake in this paper and state our main results. It also contains a discussion of the so-called Skorokhod problem with a reflective boundary condition at $0$. In Section 3 we give an overview of \cite{JB88} formulated in the context of smooth approximations of the nonsmooth systems considered in this paper, hence leading to an approximated Mortensen estimator definition. In Section 4, we prove a Hamilton-Jacobi stability result linking the stochastic filtering problem to a limit Hamilton-Jacobi equation, which is reinterpreted in this section as a value function of a control problem, but not as a value function of the Mortensen estimator for nonsmooth dynamics.

\section{Main results}
We consider the $\mathbb{R}^+$- valued state variable $x = (x(t))_{t \in [0,T]}$ solution of the variational inequality (VI)
\begin{equation}
\label{eq:dynsys}
 \forall t \mbox{ a.e.} \in [0,T], \: \forall z \geq 0, \: (f(x(t))+\omega(t) - \dot x(t))(z- x(t)) \leq 0
\end{equation}
where $f$ is a Lipshitz function from $\mathbb{R}$ to $\mathbb{R}$ and $\omega$ is the \textit{state disturbance}, that is a square integrable function from $[0,T]$ to $\mathbb{R}$. Here $x$ is continuous and differentiable almost everywhere. For adequate conditions of existence and uniqueness see \cite{brezis1973ope}. This model is classical. When $f \equiv 0$, $x$ is related to the deterministic Skorokhod problem \cite[p.231]{pardoux2014sdes} in the following way.
Given $x_0 \in \mathbb{R}^+$ and $\Omega$ the primitive of $\omega$ vanishing at $0$. 
The deterministic Skorokhod problem consists in finding a pair $(x,\Delta)$ satisfying the four conditions:
1) $x$ is a positive continuous function taking the value $x_0$ at $t=0$, 2) $\Delta$ is a continuous decreasing function vanishing at $0$, 3) $x+\Delta = x_0 + \Omega$ and 4) $\Delta$ varies only when $x=0$.
%We consider the iconic --~the so-called Skorohod’s problem~--  nonsmooth dynamical systems 
%\begin{equation}
%\label{eq:dynsys}
%\begin{cases}
%\dot \state(t) + \partial\mathcal{I}_{\RR_+} (\state(t))  \ni \modelNoise(t),\: \mbox{ a.e. } \: t > 0\\ 
%\state(0) = \initNoise
%\end{cases}
%\end{equation}
%where  $g$ is the indicator function of the convex set $\RR^+$, namely
%\begin{equation}
%\mathcal{I}_{\RR_+}(\state) :=
%\begin{cases}
%	+ \infty &\text{ if } x < 0 \\
%	0 &\text{ if } x \geq 0
%\end{cases}	
%\end{equation}
%Here, $\state(t) \in \stateSpace \simeq \RR_+$ is the state variable a time $t$, continuous and almost everywhere (a.e.) differentiable, $\modelNoise(t) \in \modelNoiseSpace \simeq \RR$ is the state \textit{disturbance}, and $\initNoise$ is the initial condition \textit{disturbance}. 
%Here $\mathcal{I}_{\RR_+} := \infty \mathbf{1}_{(-\infty,0)}$, the convex characteristic function of $[0,\infty)$ and $\partial \mathcal{I}_{\RR_+}$ is its subdifferential.
%Equivalently, \eqref{eq:dynsys} can also be formulated as a variational inequality 
%\begin{equation*}
%\label{eq:dynsys_vi}
%\begin{cases}
%\forall t \geq 0, x(t) \geq 0,\\
%\forall \: \mbox{a.e.} \: t > 0,  \forall  \xi \geq 0, \: (\modelNoise(t) - \dot \state(t))(\xi-x(t)) \leq 0,\\ 
%\state(0) = \initNoise.
%\end{cases}
%\end{equation*}
In fact, for such simple configuration of constrained dynamics, we can specify the solution: %\cite{Pardoux}~-- pas ici, si ?
\begin{equation*}
x(t) + \Delta(t) = x(0) + \int_0^t \modelNoise(s) \, \diff s \: \text{ where } \: \Delta(t) :=  \min \limits_{0 \leq s \leq t} \min \left( 0\, ;\, x(0) + \int_0^s \modelNoise(\tau) \,\diff \tau \right),
\end{equation*}
see an example of trajectory in Figure~\ref{fig:trajectory}.

\begin{figure}[htbp]
\centering
\begin{tikzpicture}
\begin{axis}[
   axis lines = middle,
   xlabel = \(t\),
   legend pos= north west,
]

 \addplot [
     domain = 0:3.7554,
     samples=50,
     color=red,
     very thick,
     ]
{5+x*cos(deg(sqrt(5)*x))+sqrt(2)*x*sin(deg(x))};

\addplot [
     domain = 3.7554:4.4080,
     samples=50,
     color=red,
     very thick,
    ]
{0};

\addplot [
     domain = 4.4080:9.5719,
     samples=50,
     color=red,
     very thick,
     ]
{4.9507+5+x*cos(deg(sqrt(5)*x))+sqrt(2)*x*sin(deg(x))};

\addplot [
     domain = 9.5719:10.1047,
     samples=50,
     color=red,
     very thick,
     ]
{0};
\addplot [
     domain = 10.1047:16,
     samples=50,
     color=red,
     very thick,
     ]
{12.3036+5+x*cos(deg(sqrt(5)*x))+sqrt(2)*x*sin(deg(x))};
\legend{$x(t)$,,,,,$x(t)+\Delta(t)$,$\Delta(t)$}

 \addplot [
     domain = 0:16,
     samples=200,
     color=black,
     ]
{5+x*cos(deg(sqrt(5)*x))+sqrt(2)*x*sin(deg(x))};

 \addplot [
     domain = 0:3.7554,
     samples=10,
     very thick,
     color=gray,
     %dasheddotted,
     ]
{0};

\addplot [
     domain = 3.7554:4.4080,
     samples=10,
     very thick,
     color=gray,
     ]
{5+x*cos(deg(sqrt(5)*x))+sqrt(2)*x*sin(deg(x))};

\addplot [
     domain = 4.4080:9.5719,
     samples=10,
     very thick,
     color=gray,
     ]
{-4.9507};

\addplot [
     domain = 9.5719:10.1047,
     samples=10,
     very thick,
     color=gray,
     ]
{5+x*cos(deg(sqrt(5)*x))+sqrt(2)*x*sin(deg(x))};
\addplot [
     domain = 10.1047:16,
     samples=10,
     very thick,
     color=gray,
     ]
{-12.3036};
\end{axis}
\end{tikzpicture}
\caption{An example of trajectory with an oscillating $\Omega$. 
We observe $x_0 + \Omega$ in black, $\Delta$ in gray and $x$ in red. 
The "upward push" $-\Delta$ keeps the resulting state variable $x$ positive. The push occurs only when $x=0$.  
}
\label{fig:trajectory}
\end{figure}
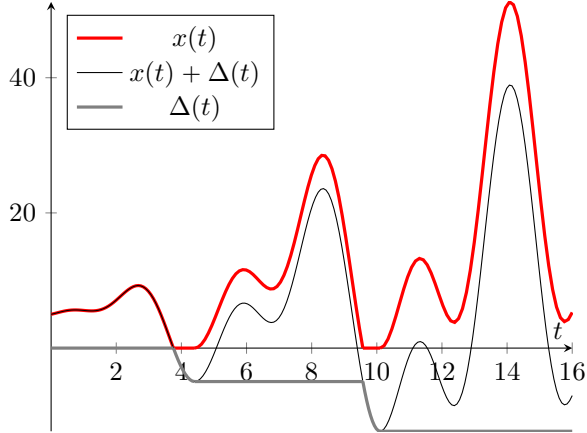

We also consider a measurement procedure $\obsMap \in C^2(\RR^+,\RR)$, so that the observations associated with a trajectory of the dynamics \eqref{eq:dynsys} are given by
\begin{equation}
\label{eq:obs}
	\forall t \geq 0, \quad \dot y(t) = \obsMap(\state(t)) + \obsNoise(t),
\end{equation}
where $\obsNoise(t) \in \RR$ is the observation \textit{disturbance}. 
\begin{remark}
We follow here the usual notation convention in stochastic filtering in which the left hand side (lhs) of \eqref{eq:obs} is denoted by $\dot y$. 
In most deterministic observation problems, the lhs of \eqref{eq:obs} is denoted by $y$. 
\end{remark}
In this deterministic estimation setting, let us specify the notation of the observed trajectory with $\{ \target{\state}(t)\}_{t\geq 0}$, 
where the measurement procedure has produced the measurement $\{ \dot y(t)\}_{t\geq 0}$, fixed from now on.
In this context, $t \mapsto \dot y(t) - \obsMap(\target{\state}(t),t)$ is a measurement error. 
Both the state and observation disturbances are unknown but we assumed that they have minimal $L^2$ energy. 
Our objective is to design a causal estimator (also called observer) of the partially observed trajectory $\{\target{\state}(t)\}_{t\geq 0}$ based only on the available information $\{\dot y(t)\}_{t\geq 0}$. The observer should be understood in the sense of \cite{Krener:1998ve}, 
in particular it is a causal estimator in the sense that for all $t\geq 0$, the estimation is based on the measurements $\{\dot y(s)\}_{0 \leq s \leq t}$ only. In other words, it is non-anticipative. 

\subsection{Maximum likelihood, Cost-to-come and Mortensen filter for the variational inequality}
Consider a triple $(\omega, x_0,x_1) \in L^2(0,t) \times [0,\infty)^2$ for which there exists a continuous function $x_{|\omega,x_0,x_1}$, differentiable almost everywhere satisfying \eqref{eq:dynsys} with the state disturbance $\omega$ and $(x_{|\omega,x_0,x_1}(0),x_{|\omega,x_0,x_1}(t)) = (x_0,x_1)$.  $x_{|\omega,x_0,x_1}$ is a synthetic variable. We drop the index notation $x_{|\omega,x_0,x_1}$ to streamline the presentation.
We can associate a finite energy with such a triple
$$
\mathcal{J}(\omega,x_0;x_1,t) := \psi(x_0) + \int_0^t \ell (\omega(s),x(s),s) {\rm d} s,
$$
where $\psi : \mathbb{R}^+ \to \mathbb{R}^+$ is locally Lipschitz and
\[
 \ell(\omega,x,s) := \frac{1}{2} \| \omega \|^2 + \frac{1}{2} \| \dot y(s) - h(x) \|^2.
\]
Otherwise if the triple $(\omega, x_0,x_1)$ does not satisfy the aforementioned condition for finite energy then its energy is infinite. In the definition of the energy above, observe that $x_0$ and $x_1$ are put on an equal footing.
In terms of path, if a triple $(\omega^\star, x_0^\star,x_1^\star) \in L^2(0,t) \times [0,\infty)^2$ is the unique (one of the) minimizer(s) of $\mathcal{J}$ then the (one of the) most likely state(s) of $\target{\state}$ is  $x_{|\omega^\star,x_0^\star,x_1^\star}$.
%\subsection{Cost-to-come}
%\begin{definition}[cost-to-come]
Let us fix the terminal state $x$ at the terminal time $t$. The \emph{cost-to-come} to the point $x$ at time $t$, given the observation $\{ \dot y(s), 0 \leq s \leq t \}$, is defined by 
\begin{equation}
\label{eq:costcome}
\costcome(x,t; \dot y(.)) := \inf_{(\omega,x_0) \in \mathcal{A}_{x, t}} \mathcal{J}(\omega,x_0;x,t) 
\: 
\mbox{ where }
\:
\mathcal{A}_{x, t} := 
\left\{  (\omega, x_0) \in L^2(0,t) \times [0,\infty), \:  \mathcal{J}(\omega,x_0;x,t) < \infty \right\}.
\end{equation}
For the variational inequality \eqref{eq:dynsys}, a Mortensen estimator $t\mapsto\hat{\state}(t)$ minimizes $\costcome(x,t)$ as a function of $x$.
\subsection{An approximated Mortensen estimator from nonsmooth dynamics penalization}
We relax the boundary constraint of the synthetic state variable that appears in the energy $\costcome(x,t; \dot y(.))$. The inequality is replaced by a nonlinear equation with a drift penalising the solution whenever it takes negatives values. We then introduce a \textit{modified cost-to-come} $\costcome^{\kappa}(x,t; \dot y(.))$ whose definition is similar to $\costcome(x,t; \dot y(.))$ in \eqref{eq:costcome} except that $\mathcal{A}_{x, t}$ is replaced by 
\[
\mathcal{A}^{\kappa}_{x, t} := \left\{ (\zeta, \omega) \in \mathbb{R}^+ \times L^2(0,t),\: 
\exists \: \state^{\kappa} \text{ that satisfies } \dot  \state^{\kappa} = f^\kappa (  \state^{\kappa})  + \modelNoise, \: \mbox{ a.e. }  \text{ with } \state^{\kappa}(0)=\initNoise, \: \state^{\kappa}(t) = \state \right\}.
\]

Here $\state^{\kappa}$ is an approximate version in $\R$ of \eqref{eq:dynsys} where
\begin{equation} 
\label{eq:pen_det}
\begin{cases}
\dot \state^{\kappa}(t) = f^\kappa ( \state^{\kappa}(t) )  + \modelNoise(t),\: \mbox{ a.e. } \: t > 0,\\ 
\state^{\kappa}(0) = \initNoise, 
\end{cases}
\end{equation}
the penalty function $f^{\kappa}$  being a $C^1$ approximation of the Moreau-Yosida regularisation $f^\kappa_0: \state \mapsto \kappa \max ( -\state , 0) + f(x)$that agrees with $f^\kappa$ over $(-\infty, -\kappa^{-1}) \cup \RR_+$ and has slope at most $-2 \kappa$ over $(-\kappa^{-1}, 0)$. 
The additional term $\kappa \max ( -\state , 0)$ vanishes as soon as $\state \geq 0$, and introduces a drift of strength $\kappa$ towards the non-negative half-line when 
$\state < -\kappa^{-1}$, and a drift of strength between $0$ and $2 \kappa$ towards the non-negative half-line when $- \kappa^{-1} < \state < 0$. As $\kappa \rightarrow + \infty$, the solution of \eqref{eq:pen_det} converges towards the solution $\state$ of \eqref{eq:dynsys} in the max norm on any finite time interval, using 
 analogous techniques than for the Moreau-Yosida regularisation  \cite{brezis1973ope}. \\
%FIN VERSION ALTERNATIVE \\
%}
We then define a variant of the Mortensen estimator with relaxation as follows:
\begin{equation}
\label{eq:mortensen}
\forall t\geq 0, \quad  \hat{\state}^{\kappa}(t) := \argmin_{\state \in \mathbb{R}} \costcome^{\kappa}(x,t; \dot y(.)),
\end{equation}
under the condition of existence and uniqueness of such a minimizer for the function $x \to \costcome^{\kappa}(x,t; \dot y(.))$. %{\color{orange}
In $\costcome^{\kappa}(x,t; \dot y(.))$, we point out that the  the given observation $\dot y(.)$ was produced --~up to measurement errors~-- from a target system $\check{\state}$ governed by a variational inequality. In other words, for the trajectory $\state^{\kappa}$ generated by the penalised dynamics the disturbance in $\dot y(.)$ contains measurement and model errors. To ease the reading, we will now write $\costcome^{\kappa}(\state,t) = \costcome^{\kappa}(x,t; \dot y(.))$.%}
%It is known as the \emph{Mortensen estimator} of the partially observed trajectory $\target{\state}^{\kappa}(t)$ through $\{ \dot y(s), 0 \leq s \leq t \} \in \Ltwo[(0,t);\obsSpace]$ \cite{Mortensen:1968vy}. 

% that is $x^{\kappa,t}$ minimizing $\costcome^{\kappa}(x,t; \dot y)$ as a function of $x$.
%\begin{figure}[htbp]
%\begin{tikzpicture}[scale=0.5]
%\begin{axis}[
%   axis lines = middle,
%   xlabel = \(x\),
%   ylabel = {\(y\)},
%]
%
% \addplot [
%     domain = -5:5,
%     samples=100,
%     color=black,
%     thin
%     ]
%{10000*(abs(x)-x)/2};
%
% \addplot [
%     domain = -5:5,
%     samples=100,
%     color=black,
%     thick
%     ]
%{50000*(abs(x)-x)/2};
%
% \addplot [
%     domain = -5:5,
%     samples=100,
%     color=black,
%     very thick
%     ]
%{100000*(abs(x)-x)/2};
%\legend{$\kappa=10000$, $\kappa=50000$, $\kappa=100000$}
%\end{axis}
%\end{tikzpicture}
%\caption{Graph of $\kappa \max ( -\state , 0)$ for three different values of $\kappa$. Note that the $y$-axis is represented $10^5$ longer than the $x$-axis.}
%\label{fig:graph-penalization}
%\end{figure}
%Exploiting the Bellman principle,
\begin{theorem}
The cost-to-come $(x,t) \mapsto \costcome^{\kappa}(\state,t)$ defined above is a viscosity solution of
\begin{equation}\label{eq:HJB-smooth}
	\begin{cases}
		\partial_t \costcome^{\kappa}(\state,t)  + \hamilton(x,t,\partial_x \costcome^{\kappa}(\state,t))  = 0, & (\state,t)\in \RR \times \RR^+ \\ 
		\costcome^{\kappa}(\state,0) = \initError(\state), & \state\in \RR
	\end{cases}
\end{equation}
where the Hamiltonian is given by 
%\PM{Notation à discuter}
\begin{equation}\label{eq:hamiltonian-min}
\hamilton(x,t,\lambda) :=  \max \limits_{\modelNoise \in \RR} \Big[ \lambda (\modelMap^\penal(\state) + \modelNoise) -  \timeError(\state,\modelNoise, t)  \Big]  =  \dfrac{1}2  \lambda^2 + \lambda  \modelMap^\penal(\state) - \dfrac{1}{2} \norm{\dot{y}(t)-h(\state)}^2.
\end{equation}
\end{theorem}
The proof is a direct adaptation of \cite{JB88}.
In Section 3, we provide further details and we recall the bridge with stochastic filtering for the problem of deterministic sequential estimation for the penalised problem in which the underlying state associated with the given observation is also governed by the penalised problem.

\subsection{Non equivalence with stochastic approaches for dynamics with variational inequality}
Inspired by the smooth dynamics approach in~\cite{JB88}, we study in Section 4 the vanishing viscosity limit of the Hamilton-Jacobi equation satisfied by the Hopf-Cole transformation of the solution of the robust Zakai equation. More precisely, we study  the Hamilton-Jacobi equation ''as is'' using the reflection method inspired in~\cite{Strauss08} to extend the Hamilton-Jacobi equations to the entire domain and avoid complications at the boundary. Since our limit $w$ of $w^\varepsilon$ when $\varepsilon$ tends to $0$ is not defined by dynamic programming, we do not obtain its boundedness in vain. This forces us to find sharper estimates of $w^\varepsilon$ than the counterparts of~\cite{JB88}. This stability result allows us to define a limit solution, which we can first understand as the value function of a control problem. However, unlike~\cite{JB88}, the limit function does not seem to follow the dynamic programming scheme of the Mortensen estimator.

%It is known as the \emph{Mortensen estimator} of the partially observed trajectory $\target{\state}^{\kappa}(t)$ through $\{ \dot y(s), 0 \leq s \leq t \} \in \Ltwo[(0,t);\obsSpace]$ \cite{Mortensen:1968vy}. 
%\textbf{Goal.} Our objective is the study of the estimation problem on the constrained dynamics \eqref{eq:dynsys}, to understand when deterministic or stochastic approach are better adapted.
%\medskip
\section{The penalised problem}
\subsection{A viscous Hamilton-Jacobi equation for the cost-to-come with penalised dynamics}\label{sec:penal-moretensen}
If we consider an optimal control pair $\initNoise$ and $\modelNoise_{|[0,t]}$ for the ``cost-to-come'' problem with terminal state $\state$ at time $t$ then for any intermediate time $t-\tau$ between the times $0$ and $t$, the part of this control enclosed by the times $0$ and $t-\tau$, namely $\modelNoise_{|[0,t-\tau]}$ remains optimal for the ``cost-to-come'' problem with terminal state $\state_{\penal|\initNoise,\modelNoise}(t-\tau)$ at time $t-\tau$. This is summarized by the following theorem proved in \cite{JB88}.
\begin{theorem}[Bellman's principle] \label{PrincipleBell}
Let $0 \leq t_{1} \leq t_{2} \leq t$, and choose $(\initNoise, \modelNoise) \in \mathcal{A}^\penal_{x, t}$. Then, we have
\[
\costcome^{\kappa} \left(\state_{\penal|\initNoise,\modelNoise}(t_{2}), t_{2}\right) \leq \costcome^{\kappa} \left(\state_{\penal|\initNoise,\modelNoise}(t_{1}), t_{1}\right)+\int_{t_{1}}^{t_{2}} \timeError\left(\state_{\penal|\initNoise,\modelNoise}(s), \modelNoise(s), s\right) \diff s.
\]
where 
$\dot \state_{\penal|\initNoise,\modelNoise} = \modelMap^\penal(\state_{\penal|\initNoise,\modelNoise}) + \omega$.
\end{theorem}
We here want to emphasize the importance of the reversibility in-time of the penalised problem to properly define the cost-to-come. Indeed, we can consider $x^{\kappa}_{\textup{rev}} : \tau \mapsto x^{\kappa}(t-\tau)$ following the dynamics
$- \dot x^{\kappa}_{\textup{rev}} (\tau) = \modelMap^{\kappa}(x^{\kappa}_{\textup{rev}} (\tau)) + \omega (\tau)$ with $x_{\textup{rev}} (0) = x.$ In this way, we find that $\mathcal{A}^{\kappa}_{x, t} \neq \emptyset$ and 
$\mathcal{A}^{\kappa}_{x, t} = \bigcup \limits_{\omega \in L^2(0,t)} \{ (x_{\textup{rev}}^{\kappa}(t),\omega) \}.$
The infinitesimal version of Bellman's principle above becomes \eqref{eq:HJB-smooth}. 
For the sake of completeness, we here recall the classical definition of a viscosity solution in $\R$. 
\begin{definition}
	Let $\mathcal{U} \in \Czero[\RR^\NstateSpace \times (0, T); \RR] .$ We say that $\mathcal{U}$ is a \emph{viscosity subsolution} of \eqref{eq:HJB-smooth} provided that for all $\testvisc \in \Cone[\RR^\NstateSpace \times (0, T); \RR]$, if $\mathcal{U} -\testvisc$ attains a local maximum at $(\state, t)$ then
\begin{equation}\label{eq:subsol-dyn-ineq}
	\partial_t \testvisc(x, t)+\hamilton(\state, t, \partial_x \testvisc(x, t)) \leq 0.
\end{equation}
We say that $\mathcal{U} $ is a viscosity supersolution of \eqref{eq:HJB-smooth} provided that for all $\testvisc \in \Cone[\RR^\NstateSpace \times (0, T); \RR]$, if $\mathcal{U} -\testvisc$ attains a local minimum at  $(\state, t)$, then
\begin{equation}\label{eq:supersol-dyn-ineq}
	\partial_t  \testvisc(x, t)+\hamilton(\state, t, \partial_x \testvisc(x, t)) \geq 0.	
\end{equation}
If $\mathcal{U}$ is both a viscosity subsolution and supersolution, we say that $\mathcal{U}$ is a \emph{viscosity solution} of \eqref{eq:HJB-smooth}.
\end{definition}
\subsection{The deterministic estimator of the penalised problem seen as the limit of a stochastic filtering problem} \label{sec:pensto}
Consider the problem of estimating the deterministic state resulting from the penalised problem given an observation. The unobserved observed trajectory is $\{ \target{\state}^\kappa(t)\}_{t\geq 0}$, where the measurement procedure has produced the measurement $\{ \dot y^\kappa(t) \}_{t\geq 0}$, fixed in this section.

Then the standard Mortensen estimator $\hat x^{\kappa,t}$ minimizes the cost to go $\mathcal{V}^{\kappa}(x,t;y^\kappa(.))$ whose definition is completely similar to 
$\mathcal{V}^{\kappa}(x,t;y(.))$ except that $y(.)$ is replaced by $y^\kappa(.)$. One alternative to solve this deterministic problem is to use a bridge with stochastic 
filtering as introduced in \cite{Hijab:1980,Hijab:1982ux} and further developped 
in \cite{JB88}. We introduce a small noise amplitude $\varepsilon > 0$, together with 
the nonlinear filtering problem in $\RR$

%\red{MISSING PARGRAPH WITH SOME MATHCAK STUFF -- DID NOT COMPILE}
%{
%\begin{verbatim}
%Then the standard Mortensen's estimator $\hat x^{\kappa,t}$ minimizes the cost to go 
%$\mathcak{V}^{\kappa}(x,t;y^\kappa(.))$ whose definition is completely similar to 
%$\mathcak{V}^{VI,\kappa}(x,t;y(.))$ except that $y(.)$ is replaced by $y^\kappa(.)$. 
%One alternative to solve this deterministic problem is to use a bridge with stochastic 
%filtering as introduced in \cite{Hijab:1980,Hijab:1982ux} and further developped 
%in \cite{JB88}. We introduce a small noise amplitude $\varepsilon > 0$, together with 
%the nonlinear filtering problem in $\R$
%\end{verbatim}
%}

\begin{equation}
\label{penalization}
%\tag{\textit{pen.$\kappa$}}
\begin{cases}
 \textup{d} X_t^{\kappa,\varepsilon} = f^{\kappa}(X_t^{\kappa,\varepsilon}) \textup{d}t + \sqrt{\varepsilon} \textup{d} B_t^1,\\
 \textup{d} Y_t^{\kappa,\varepsilon} = h(X_t^{\kappa,\varepsilon}) \textup{d}t + \sqrt{\varepsilon} \textup{d} B_t^2,\\
 \textup{with the initial condition} \: (X_0^{\kappa,\varepsilon},Y_0^{\kappa,\varepsilon}) = (\xi,0).
\end{cases}
\end{equation}
for independent brownian motions $(B^1_t)_{t\geq0}$ and $(B^2_t)_{t\geq0}$. To give a rigorous meaning to this, consider $\Omega := \mathcal{C}_0([0,\infty); \mathbb{R}^2)$ endowed with the topology of uniform convergence on compact sets. Let $\mathcal{F}$ denote the Borel $\sigma$-field on $\Omega$. For each $t \geq 0$ and $\omega \in \Omega$, define $B_t(\omega) := \omega(t)$ and set $\mathcal{F}_t := \sigma \{ B_s, \: 0 \leq s \leq t \}$ (the $\sigma$ algebra generated by $B$ up to time $t$). In this way, for all $0 \leq s \leq t, \: \mathcal{F}_s \subseteq \mathcal{F}_t$ and $\mathcal{F} = \sigma \left ( \cup_{\tau \geq 0} \mathcal{F}_\tau \right )$. We complete the triple $(\Omega,\mathcal{F}, \{ \mathcal{F}_t \})$ with the Wiener measure $\mathbb{P}$. We recall that the Wiener measure (see Karatzas \& Shreve 1991) is the unique probability measure on $(\Omega,\mathcal{F})$ satisfying for all $0 \leq s \leq t$ and $\Gamma \in \mathcal{B}(\mathbb{R}^2)$,
\[ \mathbb{P} \left ( B_t \in \Gamma | \mathcal{F}_s \right ) = \frac{1}{2\pi(s-t)} \int_\Gamma \exp \left ( -\frac{\| y - B_s \|^2}{2 (t-s)} \right ) \textup{d} y. \]
Here $\forall \zeta = (\zeta_1,\zeta_2) \in \mathbb{R}^2,\: \| \zeta \|^2 := \zeta_1^2 + \zeta_2^2$. Note that since $\{ B_0 =  0\} = \Omega$, we have $\mathbb{P}(B_0 =  0) =1$. Consider now $\varepsilon >0$, a state $\xi \geq 0$ and %two functions 
 a continuous bounded function $h: \mathbb{R} \to \mathbb{R}$, which admits a continuous bounded derivative. To assign a meaning to \eqref{penalization}, consider the mapping $\omega(.) \to (x^\kappa(.),y^\kappa(.))$ from $\mathcal{C}_0([0,T]; \mathbb{R}^2)$ to $\mathcal{C}([0,T]; \mathbb{R}^2)$ where for every $t\geq 0$,
\begin{equation*}
\begin{cases}
\displaystyle x^{\kappa,\varepsilon}(t) = \xi + \int_0^t f^{\kappa}(x^{\kappa,\varepsilon}(s)) \textup{d} s + \sqrt{\varepsilon} \omega^1(t),\\
\displaystyle y^{\kappa,\varepsilon}(t) = \int_0^t h(x^{\kappa,\varepsilon}(s)) \textup{d} s + \sqrt{\varepsilon} \omega^2(t).
\end{cases}
\end{equation*}
is well defined and continuous. If we denote this continuous map by $\phi^{\kappa,\varepsilon}_\xi$ 
then $\mathbb{P} \left( \phi^{\kappa,\varepsilon}_\xi \right)^{-1}$, the push forward measure of $\mathbb{P}$ by $\phi_\xi$, is the pathwise law associated with $(X^{\kappa,\varepsilon}, Y^{\kappa,\varepsilon})$ solving \eqref{penalization}. The filtering problem now aims to compute the measure-valued process $\left( \pi^{\kappa,\varepsilon}_t \right)_{t \geq 0}$ defined as
\[ \int_{\RR} \varphi \, \diff \pi^{\kappa,\varepsilon}_t := \mathbb{E} \left[ \varphi (X^{\kappa,\varepsilon}_t) | \sigma \left( Y^{\kappa,\varepsilon}_s \right)_{0 \leq s \leq t} \right], \]
for any bounded continuous $\varphi : \RR_+ \rightarrow \RR$, $\sigma \left( Y^{\kappa,\varepsilon}_s \right)_{0 \leq s \leq t}$ being the $\sigma$-algebra generated by the observation $Y^{\kappa,\varepsilon}_s$ up to time $t$. This estimate of $\varphi (X^{\kappa,\varepsilon}_t)$ is optimal in the least-square sense, providing the knowledge of $Y^{\kappa,\varepsilon}_s$ up to time $t$. An evolution non-linear equation called be the Kushner-Stratonovich equation can be derived for $\pi^{\varepsilon}_t$ using a sophisticated representation formula involving the \emph{innovation process}, see for instance \cite{bain2008fundamentals}. Let's focus on a rather simple approach which relies on the \emph{unnormalized conditional measure} \cite{bain2008fundamentals}
\[ \int_{\RR} \varphi \, \diff \rho^{\kappa,\varepsilon}_t := \mathbb{E} \left[ \exp \left. \left[ \frac{1}{\sqrt{\varepsilon}}\int_0^t h(x_s) \dd y_s - \frac{1}{2 \varepsilon} \int_0^t h^2(x_s) \dd s \right] \varphi (X_t) \right| \sigma \left( Y_s \right)_{0 \leq s \leq t} \right] \, , \]
which can be linked to $\pi^{\kappa,\varepsilon}_t$ by the Kallianpur-Striebel formula: for any continuous bounded function~$\varphi$
\[ \int_{\RR} \varphi \, \diff \pi^{\kappa,\varepsilon}_t = \dfrac{\displaystyle \int_{\RR} \varphi \, \diff \rho^{\kappa,\varepsilon}_t}{\displaystyle \int_{\RR} \diff \rho^{\kappa,\varepsilon}_t}. \]
This formula in this case is an analogous of Bayes formula, see \cite{kallianpur1968estimation,10.1007/BFb0085168,bain2008fundamentals}. The density $q^{\kappa,\varepsilon}(x,t)$ of $\rho^{\kappa,\varepsilon}_t$ with respect to the Lebesgue measure solves the linear stochastic partial differential equation (SPDE)
\begin{equation}
\begin{cases}
\mathrm dq^{\kappa,\varepsilon}(x,t) = A^\ast_{\kappa,\varepsilon}q^{\kappa,\varepsilon}(x,t)+\frac{1}{\varepsilon}h(x)q^{\kappa,\varepsilon}(x,t)\mathrm dY^{\kappa,\varepsilon}_t, &(x,t) \in \RR \times \RR^+ \\
q^{\kappa,\varepsilon}(x,0) = q_0^{\kappa,\varepsilon}(x), &x\in \RR
\end{cases}
\end{equation}
This is the Zakai equation, to which a rigours meaning is given in \cite{zakai1969optimal,Pardouxt1980StochasticPD,bain2008fundamentals}. The operator $A^*_{\kappa,\varepsilon}$ is the formal $L^2$ adjoint of 
\[A_{\kappa,\varepsilon}=\frac{\varepsilon}{2} \partial^2_{xx}+f^{\kappa} \partial_x,\]
The asymptotic behavior of $q^{\kappa,\varepsilon}(x,t)$ is studied in \cite{JB88} as $\varepsilon\to 0$. Instead of directly dealing with the Zakai equation, they performed the transform \cite{doss1977liens,sussmann1978gap}
\begin{equation} \label{HopfCole}
p^{\kappa,\varepsilon}(x,t)=\mathrm{exp}\Big(-\frac{1}{\varepsilon}y(t)h(x)\Big)q^{\kappa,\varepsilon}(x,t),
\end{equation}
for a given realisation $(\obsInteg(t))_{0\leq t\leq T}$ of $\left(Y^{\kappa,\varepsilon}_t\right)_{0\leq t\leq T}$, which leads to the robust form of Zakai equation \cite{clark2005robust,davis1980multiplicative,banek1986filtering}: 
\begin{equation}
\begin{cases}
\partial_t p^{\kappa,\varepsilon}(x,t)-\dfrac{\varepsilon}{2} \partial^2_{xx} p^{\kappa,\varepsilon}(x,t)+g^{\kappa}(x,t) \partial_x p^{\kappa,\varepsilon}(x,t) + \dfrac{1}{\varepsilon}\potential^{\kappa,\varepsilon}(x,t)p^{\kappa,\varepsilon}(x,t) = 0, &(x,t)\in\RR\times \R^+,
\\[0.2cm]
p^{\kappa,\varepsilon}(x,0)=q_0^{\kappa,\varepsilon}(x), & x\in\RR,
\end{cases}
\end{equation}
where $g^{\kappa}(x,t)=f^{\kappa}(x)-y(t) h'(x)$ and
\[
    \potential^{\kappa,\varepsilon}(x,t)=\frac{1}{2}h^2(x)+y(t)A_{\kappa,\varepsilon}h(x)-\frac{1}{2}y^2(t)|h'(x)|^2+\varepsilon \partial_x (f^{\kappa}(x)-y(t)h'(x)).
\]
Detailed computations can be found in appendix \ref{zakai}. By the logarithmic transformation --~also known as Hopf-Cole transform~--
\begin{equation}
\entropy^{\kappa,\varepsilon}(x,t)=-\varepsilon\log p^{\kappa,\varepsilon}(x,t),
\end{equation}
the robust form of Zakai equation can be converted into a Hamilton-Jacobi equation on $\entropy^{\kappa,\varepsilon}(x,t)$
\begin{equation}\label{belleq}
\begin{cases}
\partial_t \entropy^{\kappa,\varepsilon}(x,t) + \hamilton^{\kappa,\varepsilon}(x,t,\partial_x \entropy^{\kappa,\varepsilon})=\dfrac{\varepsilon}{2} \partial^2_{xx} \entropy^{\kappa,\varepsilon}, & (x,t)\in\RR \in \R^+,\\
\entropy^{\kappa,\varepsilon}(x,0)= \entropy^{\kappa}_0(x), & x\in\RR,
\end{cases}
\end{equation}
where
\begin{equation*}
\hamilton^{\kappa,\varepsilon}(x,t,\lambda)=\lambda g^{\kappa}(x,t)+\frac{1}{2} \lambda^2-\potential^{\kappa,\varepsilon}(x,t).
\end{equation*}
The $\varepsilon\to 0$ limit of $q^{\kappa,\varepsilon}(x,t)$ is then obtained by studying the one of $\entropy^{\kappa,\varepsilon}(x,t)$. The limit function $\entropy^{\kappa}(x,t)$ formally satisfies the Hamilton-Jacobi equation
\begin{equation}\label{HJBeq}
\begin{cases}
\partial_t \entropy^{\kappa}(x,t)+ \hamilton^{\kappa}(x,t,\partial_x \entropy^{\kappa})=0, &(x,t)\in\RR \times\R^+, \\[0.15cm]
\entropy^{\kappa}(x,0)=\entropy^{\kappa}_0(x), &x\in\RR, 
\end{cases}
\end{equation}
where 
\begin{align*}
\hamilton^{\kappa}(x,t,\lambda) &=\lambda g^{\kappa}(x,t)+\dfrac{1}{2} \lambda^2-\potential^{\kappa}(x,t),\\[0.1cm]
\potential^{\kappa}(x,t)&=\dfrac{1}{2}h^2(x)+y(t)h'(x)f^{\kappa}(x)-\dfrac{1}{2}y^2(t)|h'(x)|^2.
\end{align*}
In \cite{JB88}, the authors then establish a link between stochastic and deterministic estimation by proving that
\[ \costcome^{\kappa}(x,t) = \entropy^{\kappa}(x,t)-y(t)h(x), \]
using a uniqueness result for the vanishing viscosity solutions of \eqref{HJBeq}. 
Recall here that $\costcome^{\kappa}$ is the value function \eqref{eq:costcome} defined above and the initial condition $\psi(x) = \entropy^{\kappa}_0(x) - y(0)h(x)$.
As a by-product, they obtained the following asymptotic approximation
\[ q^{\kappa,\varepsilon}(x,t) \approx \exp \left[ - \frac{1}{\varepsilon} \costcome^{\kappa} (x,t) \right],\: \mbox{ as } \: \varepsilon \downarrow 0. \]

\subsection{The stochastic filtering problem for the constrained dynamics}

Fortunately, the stochastic filtering framework provides a way to extend the previous results to the limit case $\kappa \to \infty$ where the dynamics is constrained. This provides a candidate HJB equation that can be explored to define a Mortenten estimator for variational inequality dynamics. Since the full probabilistic framework is much more complicated, we only outline the main ingredients presented in \cite{pardoux2014sdes} and we set $f=0$ for the sake of conciseness. The resulting HJB is then rigorously analyzed as such in the next section.

Following  \cite{pardoux2014sdes}, let us consider the stochastic variational inequality in $\R_+$ 
\begin{equation}
\label{svi}
%\tag{\textit{svi}}
\begin{cases}
& \forall \text{ progressively measurable process }Z, \, \forall 0 \leq s \leq t, \\
&\int_s^t \left( Z_r - X^{\varepsilon}_r \right) \left( \sqrt{\varepsilon} \diff B^1_r - \diff X^{\varepsilon}_r \right) + \int_s^t
\mathcal{I}_{\RR_+} (X^{\varepsilon}_r) \diff r \leq \int_s^t
\mathcal{I}_{\RR_+} (Z_r) \diff r , \\
& \textup{d} Y^{\varepsilon}_t = h(X^{\varepsilon}_t) \textup{d}t + \sqrt{\varepsilon} \textup{d} B_t^2,\\
& \textup{with the initial condition} \: (X_0,Y_0) = (\xi,0),
\end{cases}
\end{equation}

\noindent with $\mathcal{I}_{\RR_+}$ which denotes the convex characteristics function of $\RR_+$ (it equals $0$ within $\RR_+$ and $+\infty$ outside). Following \cite[page 239]{pardoux2014sdes}, we say that a triple $(X^{\varepsilon},Y^{\varepsilon},K^{\varepsilon})$, an $\mathbb{R}^3$ valued stochastic process, is a solution of \eqref{svi}, if the following conditions are satisfied $\mathbb{P}$ almost surely (a.s.)
\begin{enumerate}
\item
$X^{\varepsilon},Y^{\varepsilon},K^{\varepsilon}$ are progressively measurable with continuous path and $K_0 = 0$,
\item
$\forall t \geq 0, X^{\varepsilon}_t \geq 0$,
\item
$\forall T \geq 0, \| K^{\varepsilon} \|_T < \infty$,
\item
$
\forall t \geq 0, 
\:  X^{\varepsilon}_t + K^{\varepsilon}_t = \xi + \sqrt{\varepsilon} B_t^1, \: \mbox{ and }
\:  Y^{\varepsilon}_t  = \int_0^t h(X^{\varepsilon}_s) \textup{d} s + \sqrt{\varepsilon} B_t^2,
$
\item
$
\forall 0 \leq s \leq t, \: \forall z \in [0,\infty), \: \int_s^t (z-X^{\varepsilon}_r) \textup{d} K^{\varepsilon}_r \leq 0.
$
\end{enumerate}
\noindent Since the diffusion coefficients in front of $B_1$ and $B_2$ are constant, we may fix an arbitrary $\omega \in \Omega$ and regard \eqref{svi} as a deterministic problem with forcing 
$\{ (B_t^1(\omega),B_t^2(\omega)) , \: t \geq 0 \}$.

\noindent Still following \cite{pardoux2014sdes}, we say that a triple $(x^{\varepsilon},y^{\varepsilon},k^{\varepsilon})$ is a solution of the generalized Skorokhod problem $\mathcal{GS}^{\varepsilon}$, if the following conditions hold:
\begin{enumerate}
\item
$x^{\varepsilon},y^{\varepsilon},k^{\varepsilon}$ are continuous, $x^{\varepsilon}(0) = \xi$ and $k^{\varepsilon}(0) = 0$,
\item
$\forall t \geq 0, x^{\varepsilon}(t) \geq 0$,
\item
$ k^{\varepsilon} \in BV_{loc}([0,\infty); \mathbb{R})$,
\item
$
\forall t \geq 0, 
\:  x^{\varepsilon}(t) + k^{\varepsilon}(t) = \xi + \sqrt{\varepsilon} \omega^1(t), \: \mbox{ and }
\:  y^{\varepsilon}(t)  = \int_0^t h(x^{\varepsilon}(s)) \textup{d} s + \sqrt{\varepsilon} \omega^2(t),
$
\item
$
\forall 0 \leq s \leq t, \: \forall z^{\varepsilon} \in [0,\infty), \: \int_s^t (z^{\varepsilon}-x^{\varepsilon}(r)) \textup{d} k^{\varepsilon}(r) \leq 0.
$
\end{enumerate}
\begin{theorem} \cite[Theorem 4.17 page 252]{pardoux2014sdes}
Assume $h$ to be sufficiently smooth, $x_0 \in [0,\infty)$ and $m(.)$ is a continuous function with $m(0)=0$. Then the $\mathcal{GS}^{\varepsilon}(x_0,m)$ has a unique solution.
\end{theorem}
\begin{theorem}
\cite[Theorem 4.16 page 247]{pardoux2014sdes}
The mapping $(x_0,m) \mapsto (x^{\varepsilon},y^{\varepsilon}) = \mathcal{GS}^{\varepsilon}(x_0,m)$ is continuous from
$[0,\infty) \times \mathcal{C}([0,T];\mathbb{R}^d) \to \mathcal{C}([0,T];\mathbb{R}^2)$.
\end{theorem}
\begin{theorem} \cite[Theorem 4.18 page 257]{pardoux2014sdes}
The stochastic variational inequality \eqref{svi} has a unique solution $(X^{\varepsilon},Y^{\varepsilon},K^{\varepsilon})$ progressively measurable with continuous path in the sense of the definition above.
\end{theorem}

The stochastic filtering problem of reflected diffusions has been tackled in \cite{pardoux1978stochastic,pardoux1978filtrage,menaldi1982stochastic,lions1985optimal,banek1986filtering}...
As in section \ref{sec:pensto}, the unnormalized conditional density $q^{\varepsilon}(x,t)$ can be defined for the stochastic filtering problem of the constrained dynamics, and it solves the Zakai equation with boundary condition
\begin{equation} \label{ZakaiBiundary}
\begin{cases}
\dd q^{\varepsilon}(x,t) = \frac{\varepsilon}{2} \partial^2_{xx} q^{\varepsilon}(x,t) + \dfrac{ q^{\varepsilon}(x,t) }{\varepsilon} \dd Y^{\varepsilon}_t, &(x,t) \in \RR^+\times \RR^+ \\
q^{\varepsilon}(0,x) = q^{\varepsilon}_0 (x)& x \in\RR^+, \\ 
\partial_{x} q^{\varepsilon}(t,0) = 0, &t \in\RR^+,
\end{cases}
\end{equation}
for which a rigorous meaning is given in \cite{pardoux1978stochastic,Pardouxt1980StochasticPD,pardoux1982equations}. Given a realisation $(y(t))_{0\leq t\leq T}$ of $\left(Y^{\varepsilon}_t\right)_{0\leq t\leq T}$, the change of variable
\begin{equation}
p^{\varepsilon}(x,t)=\mathrm{exp}\Big(-\frac{1}{\varepsilon}y(t)h(x)\Big)q^{\varepsilon}(x,t),
\end{equation}
now leads to the robust Zakai equation with boundary condition
\begin{equation} 
\begin{cases}
\partial_t p^{\varepsilon} (x,t) -y(t) h'(x) \partial_x p^{\varepsilon} (x,t) + \dfrac{1}{\varepsilon} \potential^{\varepsilon}(x,t) p^{\varepsilon} (x,t) = \dfrac{\varepsilon}{2} \partial^2_{xx} p^{\varepsilon}(x,t), &(x,t) \in \RR^+\times \RR^+ \\ 
\frac{\varepsilon}{2} \partial_{x} p^{\varepsilon} (t,0) + \dfrac{y(t) h'(x)}{2} p^{\varepsilon} (0,t) = 0, &t\in \RR^+,
\end{cases}
\end{equation}
for a given a realisation $(y(t))_{0\leq t \leq T}$ of $\left( Y^{\varepsilon}_t \right)_{0 \leq t \leq T}$, where 
\begin{equation}
\potential^{\varepsilon}(x,t)=\frac{1}{2}h^2(x)- \frac{\varepsilon}{2} y(t) h''(x)-\frac{1}{2}y^2(t)|h'(x)|^2.
\end{equation}
Details on this derivation are given in appendix \ref{zakai}. By the Hopf-Cole transform
\begin{equation}
\entropy^{\varepsilon}(x,t)=-\varepsilon\log p^{\varepsilon}(x,t)
\end{equation}
The robust Zakai equation can be converted into the Bellman equation \eqref{HJBVisc} on $S^{\varepsilon}(x,t)$ with boundary condition
\begin{equation*} 
\begin{cases}
\partial_t \entropy^{\varepsilon} (x,t) + \hamilton_{\entropy}^{\varepsilon} \left(x,t,\partial_x \entropy^{\varepsilon}(x,t) \right) = \dfrac{\varepsilon}{2} \partial^2_{xx} \entropy^{\varepsilon}(x,t), &(x,t) \in \RR^+\times \RR^+, \\ 
\partial_{x} \entropy^{\varepsilon} (0,t) - y(t) h'(0) = 0, &t\in\RR^+,
\end{cases}
\end{equation*}
the Hamiltonian $\hamilton_{\entropy}^\varepsilon$ being defined in \eqref{eq:def_H_S_eps} as
\begin{equation*}
\hamilton_{\entropy}^\varepsilon  :
\left\{\begin{array}{rl}
\RR_+ \times \RR_+ \times \RR &\to \RR \\
(x,t,\lambda) &\mapsto \frac{\lambda^2}{2} - \lambda y(t) h'(x) -\potential^{\varepsilon}(x,t),
\end{array}\right.
\end{equation*}

%\section{A viscous Hamilton-Jacobi equation for the limit problem} \label{sec:HJB+}
\section{A vanishing viscosity procedure for the limit problem} \label{sec:HJB+}

%{\color{orange}
For the sake of generality, we assume in this section that $f \neq 0$, with only $f(0) = 0$ to avoid additional technical problems at the boundary --~see Remark~\ref{rem:boundary-extension} for comments on the completely general case.%} 
We assume that $f$ and $y$ are bounded $C^1$ functions with bounded first derivatives, and $h$ is a bounded $C^2$ function with bounded derivatives up to order $2$ . %Some technical computations allow to extend our results to more general $f$. 
 %Extensions to space dimensions higher than $1$ could be very illustrative, but require a specific treatment of the lack of boundary regularity at $0$ that we will not address here. Related techniques can be found in~\cite{DI91corners}. 

Starting from the stochastic filtering problem of the constrained dynamics and inspired by~\cite{JB88}, we introduce the Hamilton-Jacobi equation~\eqref{HJBVisc} formally satisfied by the Hopf-Cole transform of the solution of the robust Zakai equation as done in the previous section. 
%Inspired by~\cite{JB88}, we introduce the viscous Hamilton-Jacobi equation~\eqref{HJBVisc} formally satisfied by the Hopf-Cole transform of the solution of the robust Zakai equation associated to the stochastic filtering problem of the constrained dynamics. 
We prove a stability result that allows us to recover, in the vanishing viscosity limit, what we will interpret in section~\ref{ssec:control_interpretation} as a deterministic limit of the stochastic filtering problem. Consider %Contrary to~\cite{JB88}, our approach does not rely on any previous result on the Zakai equation.
%Starting from the stochastic filtering problem of the constrained dynamics, let's try to adapt Baras and James approach to recover a deterministic limit from a viscous Hamilton-Jacobi equation. By analogy \PM{dire pourquoi on trouve Neumann par analogy} with section \ref{sec:pensto}, consider the Hamilton-Jacobi equation with boundary condition

\begin{equation}	\label{HJBVisc}		%\label{eq:S_epsilon}
\left\{\begin{aligned}
\partial_t \entropy^\varepsilon (x,t) + \hamilton_{\entropy}^\varepsilon(x,t, \partial_x \entropy^\varepsilon (x,t)) & = \frac{\varepsilon}{2} \partial_{xx}^2 \entropy^\varepsilon(x,t), \qquad & x \in \RR_+^*, & \;\;  t > 0, \\
- \partial_x \entropy^\varepsilon (0,t) & = - y(t) h'(0), \qquad & x = 0, & \;\;  t > 0, \\
\entropy^\varepsilon(x,0) & = S_0(x), \qquad & x \in \RR_+, & \;\;  t = 0,
\end{aligned} \right.
\end{equation}
%\begin{equation} \label{HJBVisc}
%\left\{\begin{array}{ll}
%\partial_t S^{\varepsilon} (x,t) + \hamilton_{\entropy}^{\varepsilon} \left(x,t,\partial_x S^{\varepsilon}(x,t) \right) = \frac{\varepsilon}{2} \partial^2_{xx} S^{\varepsilon}(x,t) \\ 
%\partial_{x} S^{\varepsilon} (0,t) - y(t) h'(0) = 0
%\end{array}
%\right.
%\end{equation}
for some initial condition $S_0 \in {\rm{BUC}}(\RR_+; \RR)$, the Hamiltonian $\hamilton_{\entropy}^\varepsilon$ being defined for $\varepsilon>0$ as 
%\begin{equation}\label{eq:def_H_S_eps}
%\hamilton_{\entropy}^\varepsilon  :
%\left\{\begin{array}{rl}
%\RR_+ \times \RR_+ \times \RR &\to \RR \\
%(x,t,\lambda) &\mapsto \frac{\lambda^2}{2} - \lambda y(t) h'(x) -\potential^{\varepsilon}(x,t),
%\end{array}\right.
%\end{equation}
\begin{equation}\label{eq:def_H_S_eps}
\hamilton_{\entropy}^\varepsilon  : \left\{
		\begin{array}{rl}
			\RR_+ \times \RR_+ \times \RR & \!\!\!\to\!\!\! \quad  \RR \\
			(x,t,\lambda) & \!\!\!\mapsto\!\!\! \quad  \dfrac{\lambda^2}{2} + \lambda g_\entropy(x,t) - \left[ \dfrac{h(x)^2}{2} + y(t) L_\varepsilon h(x) - \dfrac{1}{2} y(t)^2 |h'(x)|^2 + \varepsilon \partial_x g_\entropy (x,t) \right], % (b - (y^T Dh)^T) \right],
		\end{array}\right.
\end{equation}
where by analogy with~\cite{JB88}, we set
\begin{equation}\label{eq:JB88DefgLV}
\left|\,
\begin{aligned}
g_\entropy (x,t) & := f(x) - y(t) h'(x), \\
L_\varepsilon & := \frac{\varepsilon}{2} \partial_{xx}^2 + f(x) \partial_x. %, \\
%V_\entropy^\varepsilon(x,t) & := \left[ \frac{h(x)^2}{2} + y(t) L_\varepsilon h(x) - \frac{1}{2} y(t)^2 |\partial_x h(x)|^2 + \varepsilon \partial_x g_\entropy (x,t) \right].
%L_\varepsilon h(x) & = \left( L_\varepsilon h_1(x), \dots, L_\varepsilon h_m(x) \right)^T.
\end{aligned} \right.
\end{equation}

Contrary to James and Baras~\cite{JB88} who started from the stochastic setting, this deterministic equation will be our starting point, not requiring any previous result on the robust Zakai equation with boundary conditions, and defining $\entropy^\varepsilon$ as solution of~\eqref{HJBVisc} rather than as the value function resulting from a dynamic programming approach.
%, for which we are going to provide well-posedness. By reversing the Hopf-Cole transform \eqref{HopfCole}, this approach would gives well-posedness for the related robust Zakai equation with boundary condition.

%Contrary to Baras and James~\cite{JB88}, who started from the stochastic setting, this deterministic equation will be our starting point, for which we are going to provide well-posedness. By reversing the Hopf-Cole transform \eqref{HopfCole}, this approach would gives well-posedness for the related robust Zakai equation with boundary condition.

\begin{remark}\label{rem:boundary-extension}
If $f(0) \neq 0$, the second line of equation~\eqref{HJBVisc} reads instead
\[
- \partial_x \entropy^\varepsilon (0,t) = - y(t) h'(0) - 2 f(0), \qquad x = 0, \;\;  t > 0.
\]
We may add an appropriate smooth, bounded perturbation of bounded derivatives to $\entropy^\varepsilon$, defining for instance:
\[
 \bar \entropy^\varepsilon (x,t) := \entropy^\varepsilon (x,t) - 2 x f(0) e^{-x^2},
\]
so that $- \partial_x \bar \entropy^\varepsilon (x,t) = - \partial_x \entropy^\varepsilon (x,t) + 2 f(0)$. We thus recover a function satisfying a closely related viscous Hamilton-Jacobi equation whose Hamiltonian can be easily computed. That new Hamiltonian satisfies the same sufficient properties for the rest of the section, and the boundary condition of the new equation does not involve $f$. Hence, similar results will hold, so to avoid unnecessary technicalities, we choose to take $f(0) = 0$ hereafter.
\end{remark}

\subsection{Viscous Hamilton-Jacobi equation on $\entropy^\varepsilon$}

We denote the formal limit of $\hamilton_{\entropy}^\varepsilon$ as $\varepsilon \to 0$ by
\begin{equation}\label{eq:def_H_S_1d}
\hamilton_{\entropy}: \left\{
\begin{array}{ll}
  \RR_+ \times \RR_+ \times \RR & \to \quad \RR \\
 (x,t,\lambda) & \mapsto \quad  \frac{\lambda^2}{2} + \lambda g_\entropy (x,t) - \frac{(h(x))^2}{2} - y(t) f(x) h'(x) + \frac{1}{2}(y(t))^2 |h'(x)|^2.
\end{array}\right.
\end{equation}
%\begin{equation}\label{eq:def_H_S_1d}
%\hamilton_{\entropy}:\left\{
%\begin{array}{ll}
%\RR_+ \times \RR_+ \times \RR &\to \RR \\
%(x,t,\lambda) &\mapsto \frac{\lambda^2}{2} - \lambda y(t) h'(x) - \frac{h^2(x)}{2} + \frac{1}{2}y^2(t) |h'(x)|^2.
%\end{array}
%\right.
%\end{equation}
The main theorem of the section is the following stability result. %theorem is the following one.

\begin{theorem} \label{thm:S}
Assume that $f, y \in C^1_b(\RR_+; \RR)$ and $h \in C^2_b(\RR_+; \RR)$ are bounded with first (and second for $h$) bounded derivatives, and $S_0 \in {\rm BUC}(\RR_+; \RR)$. Then:
\begin{description}
\item[$(i)$]
for all $\varepsilon > 0$ the second order evolution Hamilton-Jacobi equation
\begin{equation}\label{eq:S_epsilon} %\tag{$HJ^\varepsilon$}
\left\{\begin{aligned}
\partial_t \entropy^\varepsilon (x,t) + \hamilton_{\entropy}^\varepsilon(x,t, \partial_x \entropy^\varepsilon (x,t)) & = \frac{\varepsilon}{2} \partial_{xx}^2 \entropy^\varepsilon (x,t), \qquad & x \in \RR_+^*, & \;\;  t > 0, \\
- \partial_x \entropy^\varepsilon (0,t) & = - y(t) h'(0), \qquad & x = 0, & \;\;  t > 0, \\
\entropy^\varepsilon(x,0) & = S_0(x), \qquad & x \in \RR_+, & \;\;  t = 0,
\end{aligned} \right.
\end{equation}
admits a unique smooth solution $S^{\varepsilon}$.
\item[$(ii)$]
The Hamiltonian $\hamilton_{\entropy}^\varepsilon$ defined in~\eqref{eq:def_H_S_eps} converges locally uniformly as $\varepsilon \to 0$ to the limiting Hamiltonian $\hamilton_{\entropy}$ of equation~\eqref{eq:def_H_S_1d}, 
\item[$(iii)$]
$\entropy^\varepsilon$ converges locally uniformly as $\varepsilon \to 0$ to a continuous function we denote by $\entropy$, 
\item[$(iv)$]
$\entropy$ is the unique viscosity solution of the limiting Hamilton-Jacobi equation below, in the sense of Definition~\ref{def:visc_sol_HJ}.
\begin{equation}\label{eq:S} %\tag{$HJ$}
\left\{\begin{aligned}
\partial_t \entropy (x,t) + \hamilton_{\entropy}(x,t, \partial_x \entropy (x,t)) & = 0, \qquad & x \in \RR_+^*, & \;\;  t > 0, \\
- \partial_x \entropy (0,t) & = - y(t) h'(0), \qquad & x = 0, & \;\;  t > 0, \\
\entropy(x,0) & = \entropy_0(x), \qquad & x \in \RR_+, & \;\;  t = 0.
\end{aligned} \right.
\end{equation}
\end{description}
\end{theorem}

Let us recall from \textit{e.g.} \cite{Lions:1985ve, Barles93} an appropriate notion of solution for the above Hamilton-Jacobi equations with Neumann boundary condition. Consider the first order Hamilton-Jacobi equation on $\RR_+$
\begin{equation}\label{eq:HJ_generic}
\left\{\begin{aligned}
\partial_t u(x,t) + H(x,t,u(x,t),\partial_x u(x,t)) = 0, & \quad x \in \RR_+^*, t > 0, \\
B(0,t, u(0,t), \partial_x u(0,t)) = 0 & \quad x = 0, t > 0, \\
%-\partial_x u(0,t) = B(t) & \quad t >0, \\
u (x,t) = u_0 (x), & \quad x \in \RR_+, t = 0,
\end{aligned}\right.\end{equation}
for locally Lipschitz $H$ and $B$, the latter being strictly increasing with respect to its last variable in the outward normal direction at $x$:  for all $R > 0$, there exists $\nu_R > 0$ such that for all $(x, t, u, \lambda) \in \{0\} \times \RR_+ \times [-R, R] \times \RR$,
\begin{equation}\label{eq_Barles93_4}
B(x,t, u, \lambda + \alpha n(x)) - B(x,t,u,\lambda) \geq \nu_R \alpha,
\end{equation}
where $n(x)$ is the unit outward normal to $\partial \RR_+$ at $x$ -- so, $-1$. Note that the satisfaction of this condition is the reason for the $-$ sign preceding $\partial_x \entropy(0,t)$ in~\eqref{eq:S}, in which:
\[
B(x,t, \lambda) =  - \lambda + y(t) h'(0).
\]

\begin{definition} \label{def:visc_sol_HJ}
A continuous function $u$ is said to be a viscosity subsolution (respectively, supersolution) of equation~\eqref{eq:HJ_generic} if it satisfies that for all $\phi \in C^1( \RR_+ \times \RR_+; \RR)$, at each maximum point $(x_0, t_0) \in \RR_+ \times \RR_+$ of $u-\phi$, we have:
%\begin{equation*}
%\left\{\begin{aligned}
%\text{If } (x_0, t_0) \in \RR_+^* \times \RR_+^*, & \qquad
%	(\partial_t \phi + H(\cdot, u, \partial_x \phi)(x_0,t_0) \leq 0, \\
%\text{If } (x_0, t_0) \in \{0\} \times \RR_+^*, & \qquad \min \left\{ -f(t_0) - \partial_x \phi (0,t_0) \; , \;  \partial_t (\phi + H(\cdot, u, \partial_x \phi))(0,t_0) \right\} \leq 0, \\
%\text{If } (x_0, t_0) \in R^*_+ \times \{0\}, & \qquad 
%\min \left\{ u(x_0,0) - u_0(x_0) \;, \; (\partial_t u + H(\cdot, u, \partial_x \phi))(x_0,0) \right\} \leq 0, \\
%\text{If } (x_0, t_0) = (0,0), & \qquad \min \left\{ u(0, 0) - u_0(0) \;, \; -f(0) - \partial_x \phi (0,0) \; , \; (\partial_t u + H(\cdot, u, \partial_x \phi))(0,0) \right\} \leq 0,
%\end{aligned}\right.\end{equation*}
\begin{equation*}
\left\{\begin{aligned}
\text{If } (x_0, t_0) \in \RR_+^* \times \RR_+^*, & \qquad
	(\partial_t \phi + H(\cdot, u, \partial_x \phi))(x_0,t_0) \leq 0, \\
\text{If } (x_0, t_0) \in \{0\} \times \RR_+^*, & \qquad \min \left\{ B(\cdot, u, \partial_x \phi )(0,t_0) \; , \;  (\partial_t \phi + H(\cdot, u, \partial_x \phi))(0,t_0) \right\} \leq 0, \\
\text{If } (x_0, t_0) \in R^*_+ \times \{0\}, & \qquad 
\min \left\{ u(x_0,0) - u_0(x_0) \;, \; (\partial_t \phi + H(\cdot, u, \partial_x \phi))(x_0,0) \right\} \leq 0, \\
\text{If } (x_0, t_0) = (0,0), & \qquad \min \left\{ u(0, 0) - u_0(0) \;, \;B(\cdot, u, \partial_x \phi )(0,t_0)  \; , \; (\partial_t \phi + H(\cdot, u, \partial_x \phi))(0,0) \right\} \leq 0,
\end{aligned}\right.\end{equation*}

(respectively, for all $\phi \in C^2(\bar \RR_+ \times \RR_+; \RR)$, at each minimum point $(x_0, t_0) \in \RR_+ \times \RR_+)$ of $u-\phi$, we have:
%\begin{equation*}
%\left\{\begin{aligned}
%\text{If } (x_0, t_0) \in \RR_+^* \times \RR_+^*, & \qquad
%	(\partial_t \phi +H(\cdot, u, \partial_x \phi)(x_0,t_0) \geq 0, \\
%\text{If } (x_0, t_0) \in \{0\} \times \RR_+^*, & \qquad \max \left\{ -f(t_0) - \partial_x \phi (0,t_0) \; , \;  \partial_t (\phi + H(\cdot, u, \partial_x \phi))(0,t_0) \right\} \geq 0, \\
%\text{If } (x_0, t_0) \in R^*_+ \times \{0\}, & \qquad 
%\max \left\{ u(x_0,0) - u_0(x_0) \;, \; (\partial_t u + H(\cdot, u, \partial_x \phi))(x_0,0) \right\} \geq 0, \\
%\text{If } (x_0, t_0) = (0,0), & \qquad \max \left\{ u(0, 0) - u_0(0) \;, \; -f(0) - \partial_x \phi (0,0) \; , \; (\partial_t u + H(\cdot, u, \partial_x \phi))(0,0) \right\} \geq 0,
%\end{aligned}\right.\end{equation*}
\begin{equation*}
\left\{\begin{aligned}
\text{If } (x_0, t_0) \in \RR_+^* \times \RR_+^*, & \qquad
	(\partial_t \phi + H(\cdot, u, \partial_x \phi))(x_0,t_0) \geq 0, \\
\text{If } (x_0, t_0) \in \{0\} \times \RR_+^*, & \qquad \max \left\{ B(\cdot, u, \partial_x \phi )(0,t_0) \; , \;  (\partial_t \phi + H(\cdot, u, \partial_x \phi))(0,t_0) \right\} \geq 0, \\
\text{If } (x_0, t_0) \in R^*_+ \times \{0\}, & \qquad 
\max \left\{ u(x_0,0) - u_0(x_0) \;, \; (\partial_t \phi + H(\cdot, u, \partial_x \phi))(x_0,0) \right\} \geq 0, \\
\text{If } (x_0, t_0) = (0,0), & \qquad \max \left\{ u(0, 0) - u_0(0) \;, \;B(\cdot, u, \partial_x \phi )(0,t_0)  \; , \; (\partial_t \phi + H(\cdot, u, \partial_x \phi))(0,0) \right\} \geq 0,
\end{aligned}\right.\end{equation*}

A continuous function $u$ is said to be a viscosity solution of equation~\eqref{eq:HJ_generic} if it is both a viscosity subsolution and supersolution.
\end{definition}

\begin{theorem}[Uniqueness and conditional existence of ${\rm{BUC}}$ solutions -- Theorem 2.1 in \cite{Barles93}]  \label{thm:existence_and_uniqueness}
Assume the initial condition $u_0$ to be bounded and uniformly continuous. Assume $H$ and $B$ to be locally Lipschitz continuous, and that $H$ is locally uniformly Lipschitz continuous, convex and coercive in its last variable. Then, if $u$ and $v$ are respectively a bounded u.s.c. viscosity subsolution and a bounded l.s.c. viscosity supersolution of~\eqref{eq:HJ_generic}, then
\[
u \leq v \qquad {\text{on }} \quad \bar \Omega \times [0,T].
\]
Moreover, if such $u,v$ exist and $u=v=u_0$ on $\bar \Omega \times \{0\}$, then equation~\eqref{eq:HJ_generic} admits a continuous unique viscosity solution.
\end{theorem}

\begin{remark}
There are crucial hypotheses of Barles' theorem above that become immediate in our one spatial dimension, first order Hamiltonian setting. First, the open set $\R_+^*$ trivially satisfies that $\partial \R_+^* = \{0\} \in W^{3,\infty}$. Second, the structure hypotheses labeled $(H1)$, $(H2)$, and $(H3)$ in~\cite{Barles93} are clearly satisfied by a first order Hamiltonian $H$ and by our boudary condition, and boil down to . 
\end{remark}

To consider homogeneous Neumann conditions, let's work on $w^{\varepsilon}(x,t) := S^{\varepsilon}(x,t) - y(t)h(x)$, instead of directly $S^{\varepsilon}$. $w^{\varepsilon}$ is given as the solution of 
\begin{equation}\label{eq:w_epsilon_1d}
\left\{\begin{aligned}
\partial_t w^\varepsilon (x,t) + \hamilton^\varepsilon(x,t, \partial_x w^\varepsilon (x,t)) & = \frac{\varepsilon}{2} \partial_{xx}^2 w^\varepsilon(x,t), \qquad & x \in \RR_+^*, & \;\;  t > 0, \\
- \partial_x w^\varepsilon (0,t) & = 0, \qquad & x = 0, & \;\;  t > 0, \\
w^\varepsilon(x,0) & = w_0(x), \qquad & x \in \RR_+, & \;\;  t = 0,
\end{aligned} \right.
\end{equation}
provided with some locally bounded, Lipschitz initial condition $w_0$. Local existence and uniqueness is shown in Section~\ref{sec:W_eps_local_existence_and_uniqueness} and global existence and uniqueness in Section~\ref{ssec:bounds}, Theorem~\ref{thm:w_restriction_uniform_bounds}. The Hamiltonian $\hamilton^\varepsilon$ is defined over $(x,t,\lambda) \in \RR_+ \times \RR_+ \times \RR$ as
\begin{equation}\label{eq:def_H_w_eps}
\hamilton^\varepsilon (x,t,\lambda) := \frac{1}{2} \lambda^2 + \lambda f(x) - \frac{1}{2} (h(x))^2 - \varepsilon f'(x) + \frac{\varepsilon}{2} y(t) h''(x) + \dot y(t) h(x).
\end{equation}

Note that this Hamiltonian, its $\varepsilon \rightarrow 0$ limit and all the Hamiltonians considered in this paper satisfy the hypotheses of Theorem~\ref{thm:existence_and_uniqueness}.

\begin{remark}
With this point of view, it is possible to directly define $w^\varepsilon$ as the solution of equation~\eqref{eq:w_epsilon_1d} after proving that it is well-posed, and to introduce $\entropy^\varepsilon$ as a modification of $w^\varepsilon$. This allows to consider $\entropy^\varepsilon$ without starting from the general Zakai equation. %, and thus the Hopf-Cole transform would provide well-posedness results for the related robust Zakai equation with boundary condition.
\end{remark}

\begin{proposition}[Local uniform convergence of the viscous Hamiltonian]~%
\label{prop:H_eps_CV_H_1d}~%
$\hamilton^\varepsilon$ converges uniformly to $\hamilton$ in $C^0(\RR_+ \times \RR_+ \times \RR , \RR)$, where:
\begin{equation*} \label{eq:HJ-limit}
\hamilton (x,t,\lambda) := \frac{1}{2} \lambda^2 + \lambda f(x) - \frac{1}{2} (h(x))^2 + \dot y(t) h(x).
\end{equation*}
%where $y$ here is the limit of $y^\varepsilon$. 
In an analogous way, $\hamilton_{\entropy}^\varepsilon$ defined in equation~\eqref{eq:def_H_S_eps} converges locally uniformly to $\hamilton_{\entropy}$ defined in equation~\eqref{eq:def_H_S_1d}.
\end{proposition}	

Formally, equation~\eqref{eq:w_epsilon_1d} tends to the following:
\begin{equation}\label{eq:w_limit}
\left\{\begin{aligned}
\partial_t w (x,t) + \hamilton(x,t,\partial_x w (x,t)) & = 0, \qquad & x \in \RR_+^*, & \;\;  t > 0, \\
- \partial_x w (0,t) & = 0, \qquad & x = 0, & \;\;  t > 0, \\
w(x,0) & = w_0(x), \qquad & x \in \RR_+, & \;\;  t = 0,
\end{aligned} \right.
\end{equation}
where the boundary condition must be understood in the sense of viscosity solutions, as in Definition~\ref{def:visc_sol_HJ}.%~, t
\begin{remark}
If $w$ is a viscosity solution of~\eqref{eq:w_limit}, the remark in Section 2 of~\cite{Barles93} still holds: using well-chosen test functions, it is possible to prove that the initial condition is satisfied in the classical sense provided $w_0$ is smooth, as is the case here. For an extension to nonsmooth initial conditions, we refer to the corresponding chapter of~\cite{Barles94}.
\end{remark}

\subsection{Local existence and uniqueness for the solution of~\eqref{eq:w_epsilon_1d}} \label{sec:W_eps_local_existence_and_uniqueness}

We wish to extend equation~\eqref{eq:w_epsilon_1d} to $x \in \RR$ in a way that guarantees that the restriction to $x \in \RR_+$ of the solution of the extended equation $\tilde w$ satisfies the Neumann boundary condition. Hence, it is sufficient to construct an extention $\tilde w$ that is even, so $\partial_x \tilde w$ is odd. Let us proceed by analogy with a reflection method presented in~\cite[Ch. 3]{Strauss08} for the heat equation with Neumann boundary condition:
\begin{equation*}\label{eq:HE_CN}
\left\{\begin{aligned}
\partial_t u(x,t) - k \partial_{xx}^2 u (x,t) & = F(x,t), \qquad & x > 0,  & \;\; t > 0, \\
\partial_x u (0,t) & = 0, \qquad & x = 0, & \;\;  t > 0, \\
u(x,0) & = u_0(x) \qquad & x \geq 0, & \;\; t = 0.
\end{aligned} \right.
\end{equation*}
Let $G$ be the Green heat kernel, defined over $(x,t) \in \RR \times \RR_+$ as:
\begin{equation*}\label{eq:1d_heat_kernel}
G_k(x,t) = \frac{1}{\sqrt{4 k \pi t}} \exp \left( - \frac{x^2}{4kt} \right).
\end{equation*}
The function
\[
\tilde u(x,t) := [G_k(\cdot, t) * u_0(|\cdot|)](x) + \int_0^t [G_k(\cdot, t-s)* F(|\cdot|, s)](x) \,{\rm d}s
\]
is the Duhamel formulation corresponding to the symetrised equation
\begin{equation*}\label{eq:HE_CN_whole_space}
\left\{\begin{aligned}
\partial_t \tilde u(x,t) - k \partial_{xx}^2 \tilde u (x,t) & = F(|x|,t), \qquad & x \in \RR, & \;\;  t > 0, \\
%\partial_x \tilde u (0,t) & = 0, \qquad & x = 0, & \;\;  t > 0, \\
\tilde u(x,0) & = u_0(|x|), \qquad & x \in \RR, & \;\;  t = 0,
\end{aligned} \right.
\end{equation*}
and its restriction to $x \in \RR_+$ satisfies the initial Heat equation with Neumann boundary condition.

In an analogous way, we define the symmetrised Hamiltonian $\tilde \hamilton$, taking into account that the variable $\lambda$ will be expected to be an odd function of $x$: 
\begin{equation}\label{eq:def_tilde_H_w_eps}
\tilde \hamilton^\varepsilon: \left\{
\begin{array}{ll}
\RR \times \RR_+ \times \RR & \to \RR \\
(x,t,\lambda) & \mapsto \hamilton (|x|, t, {\rm sgn}(x) \lambda) = \frac{1}{2} \lambda^2 + \lambda g(x) - \tilde V_w^\varepsilon(x,t),
\end{array}
\right.
\end{equation}
with 
\begin{equation}\label{eq:def_Vtildeeps}
\left\{ \begin{aligned}
g (x) & :=  {\rm sgn}(x) f(|x|), \\
\tilde V_w^{\varepsilon}(x,t) & := \frac{1}{2} (h(|x|))^2 - \dot y(t) h(|x|) + \varepsilon f ' (|x|) - \frac{\varepsilon}{2} y(t) h''(|x|).
\end{aligned}\right.
\end{equation}
Note that $g$ for $\tilde w^\varepsilon$ corresponds to $g_\entropy$ for $\entropy$, defined in~\eqref{eq:JB88DefgLV}. The `symmetrised' version of equation~\eqref{eq:w_epsilon_1d} reads
\begin{equation}\label{eq:tilde_w_epsilon}
\left\{\begin{aligned}
\partial_t \tilde w^\varepsilon (x,t) + \tilde \hamilton^\varepsilon(x,t, \partial_x \tilde w^\varepsilon (x,t)) & = \frac{\varepsilon}{2} \partial_{xx}^2 \tilde w^\varepsilon(x,t), \qquad & x \in \RR, & \;\;  t > 0, \\
\tilde w^\varepsilon(x,0) & = \tilde w_0(x), \qquad & x \in \RR, & \;\;  t = 0,
\end{aligned} \right.
\end{equation}
where we use $\tilde w_0 : x \in \RR \mapsto w_0(|x|)$. Note that there is no more Neumann boundary condition.

Let us establish the well-posedness of the equation above.

\begin{theorem}[Local existence and uniqueness of a solution of~\eqref{eq:tilde_w_epsilon}] \label{thm:uniqueness_eps}
Let $\varepsilon > 0$. Let $\tilde w_0 \in L^\infty \cap {\rm Lip}$, and $\tilde \hamilton^\varepsilon \in {\rm Lip}_{\rm loc} \text{w.r.t.} \; \lambda$. Then there exists $T>0$ such that there exists a unique smooth solution $\tilde w^\varepsilon$ of equation~\eqref{eq:tilde_w_epsilon} defined on $\RR \times [0,T]$.
\end{theorem}

The proof of Theorem~\ref{thm:uniqueness_eps} is a technical, but relatively standard fixed-point method, so for the sake of conciseness, we will only sketch it.
\begin{proof}
We fix $\varepsilon$, and assume, in a first step, that $\tilde \hamilton^\varepsilon$ is globally Lipschitz in $\lambda$. We prove that for $(x,t) \in \RR \times [0,T]$ with $T$ small enough, the mapping of a Picard iterate to the next is a contraction in the norm $\|u\| := \|u\|_{L^\infty} + \|\partial_x u\|_{L^\infty}$. We then extend the result to $\tilde \hamilton^\varepsilon$ locally Lipschitz in $\lambda$ by applying a security cylinder method used in~\cite[Ch.V]{Demailly16}. Smoothness follows from that of the Green kernel.
\end{proof}

%We provide the proof of Theorem~\ref{thm:uniqueness_eps} in the Appendix~\ref{ssec:A_FixedPoint}. It's a relatively standard fixed-point result, first assuming $\tilde \hamilton^\varepsilon$ is globally Lipschitz in $\lambda$, then extending the result to $\tilde \hamilton^\varepsilon$ locally Lipschitz in $\lambda$ by applying a security cylinder method. 	

\subsection{Uniform in $\varepsilon$ bounds on $w^\varepsilon$} %$W^{1,\infty}$ u
\label{ssec:bounds}

The main result of this section is a global existence, uniqueness and uniform-in-$\varepsilon$ boundedness theorem on $w^\varepsilon$:
\begin{theorem} \label{thm:w_restriction_uniform_bounds}
Equation~\eqref{eq:w_epsilon_1d} admits a unique solution $w^\varepsilon$ defined globally in time, and $w^\varepsilon$ is locally bounded in $W_x^{1,\infty} (\RR_+; C_t^{1,1})$: over any compact set $Q$ there exists $K > 0$ such that for all $(x,t), (x,s) \in Q$,
  	\begin{equation} \label{eq:w_restriction_1_infty_bounds}
  	\begin{aligned}
  	%& (i) \qquad & {\rm Lip}_x u^\varepsilon & \leq {\rm Lip}\; u_0 + t {\rm Lip}_x \hamilton, \medskip\\
  	%& (ii) \qquad & \|u^\varepsilon\|_{L^\infty_{x,T}} & \leq \|u_0\|_{L^\infty} + T \tilde A_0,
	& (i) \qquad & \left|  w^\varepsilon (x,t) \right| & \leq K, \medskip\\
  	& (ii) \qquad & \left| \partial_x  w^\varepsilon (x,t) \right| & \leq K, \medskip\\
  	& (iii) \qquad & \left|  w^\varepsilon (x,t) - w^\varepsilon (x,s) \right| & \leq K \left(\left| t-s \right|^{1/2} + |t-s| \right).
   	\end{aligned} 
  	\end{equation} 
  Moreover, $(i)$ can be refined into a sharper estimate $(iv)$, where the bound itself does not depend on $R$. Namely, for all $R \geq \max (8, 16 \|f\|_{L^\infty (\RR)})$ there exists $\varepsilon_R := \frac{1}{32 R^4}$ such that for all $0 < \varepsilon < \varepsilon_R$, 
  \begin{equation}
(iv) \qquad \| w^\varepsilon \|_{L^\infty (Q_R)} \leq \| w_0\|_{L^\infty (\RR_+)} + \left[ 8(1+\|f^\varepsilon\|_{L^\infty (\RR_+)}^2) +  \|V^\varepsilon\|_{L^\infty(\RR_+ \times [0,T])} + 1 \right] T + 1.
\end{equation} %as detailed in Theorem~\ref{thm:W_uniform_bounds}.
\end{theorem}

For the sake of simplicity, we will first prove uniform in $\varepsilon$ estimates on $\tilde w^\varepsilon$, %$u^\varepsilon := \tilde w^\varepsilon$,
the even extension to $\RR \times \RR_+$ of $w^\varepsilon$, which we defined in equation~\eqref{eq:tilde_w_epsilon}. %in the proof of Theorem~\ref{thm:extension_to_R_d} above.
The bounds on $w^\varepsilon$ and its global existence are a direct corollary of the corresponding theorems on $\tilde w^\varepsilon$.

%\begin{remark}
Our proof strategy in this section is that of James and Baras~\cite{JB88}, with the exceptions that we apply it to $\tilde w^\varepsilon$ rather than the extension of $\entropy^\varepsilon$, that we only have local existence of the solution at fixed $\varepsilon$ for now, and that we need to glean a sharper $L^\infty_{\rm{loc}}$ estimate. The global existence of the solution for each $\varepsilon$ is a consequence of the uniform bounds (Corollary~\ref{cor:tilde_w_global_existence}), and the exact same proof can then be applied over $[0,T]$.
%\end{remark}

\begin{theorem}\label{thm:W_uniform_bounds}
Assume $\tilde \hamilton$ satisfies the assumptions of Theorem \ref{thm:existence_and_uniqueness}, and let $T>0$ such that equation~\eqref{eq:tilde_w_epsilon} admits a unique smooth solution $\tilde w^\varepsilon$ over $\RR \times [0,T]$.
  	Then for every compact subset $Q \subset \RR \times [0,T]$ 
  	there exists $\varepsilon_0 >0$ and $K >0$ such that for all $0 < \varepsilon < \min (\varepsilon_0, 1)$, for all $(x,t), (x,s) \in Q$, $\tilde w^\varepsilon$ satisfies:
  	\begin{equation} \label{eq:W_1_infty_bounds}
  	\begin{aligned}
	& (i) \qquad & \left| \tilde w^\varepsilon (x,t) \right| & \leq K, \medskip\\
  	& (ii) \qquad & \left| \partial_x \tilde w^\varepsilon (x,t) \right| & \leq K, \medskip\\
  	& (iii) \qquad & \left| \tilde w^\varepsilon (x,t) - \tilde w^\varepsilon (x,s) \right| & \leq K \left(\left| t-s \right|^{1/2} + |t-s| \right).
   	\end{aligned} 
  	\end{equation} 
  Moreover, $(i)$ can be refined into a sharper estimate $(iv)$, where the bound itself does not depend on $R$. Namely, for all $R \geq 8$ there exists $\varepsilon_R := \frac{1}{32 R^4}$ such that for all $0 < \varepsilon < \varepsilon_R$, 
    \begin{equation} \label{eq:sharp_Linfty}
(iv) \qquad \| \tilde w^\varepsilon \|_{L^\infty (Q_{R})} \leq \|\tilde w_0\|_{L^\infty (\RR)} + \left[8(1+\|g\|_{L^\infty (\RR)}) + \|V^\varepsilon\|_{L^\infty(\RR \times [0,T])} + 1 \right] T + 1.
\end{equation}
%  \begin{equation}
%(iv) \qquad \| \tilde w^\varepsilon \|_{L^\infty (Q_R)} \leq \|\tilde w_0\|_{L^\infty (\RR)} + \left[ \|\tilde V_w^\varepsilon\|_{L^\infty(\RR \times [0,T])} + 1 \right] T + 1.
%\end{equation}
\end{theorem}

%In the following, $L^\infty$ norms will be taken over the wholed domain unless specified otherwise.

%Before proving the theorem above, let us state a global existence corollary on $\tilde w^\varepsilon$, and a global existence and boundedness theorem on $w^\varepsilon$ and show it is a consequence of Theorems~\ref{thm:W_uniform_bounds} and~\ref{thm:uniqueness_eps}.

\begin{corollary} \label{cor:tilde_w_global_existence}
Equation~\eqref{eq:tilde_w_epsilon} admits a unique solution $\tilde w^\varepsilon$ defined globally in time. Moreover, $\tilde w^\varepsilon$ is locally bounded in the norm of Theorem~\ref{thm:W_uniform_bounds}. 
\end{corollary}

\begin{proof}{Proof of Corollary~\ref{cor:tilde_w_global_existence} assuming Theorem~\ref{thm:W_uniform_bounds}}~%

Consider the maximal interval of existence in time of the local solution $\tilde w^\varepsilon$. The local uniform boundedness of $\tilde w^\varepsilon$ allows to prove that interval is $[0, \infty)$, which implies the existence of a unique solution $\tilde w^\varepsilon$ defined globally in time. This in turn allows to apply Theorem~\ref{thm:W_uniform_bounds} globally in time, recovering the same bounds over every compact. Uniqueness follows from Theorem~\ref{thm:uniqueness_eps}.
\end{proof}

\begin{proof}[Proof of Theorem~\ref{thm:w_restriction_uniform_bounds}]~%

Assume Theorem~\ref{thm:W_uniform_bounds} and Corollary~\ref{cor:tilde_w_global_existence} hold. Then the restriction $(x,t) \in \RR_+ \times \RR_+ \mapsto \tilde w^\varepsilon (x, t)$ is well defined globally, bounded locally, and satisfies the equation~\eqref{eq:w_epsilon_1d}. Here too, uniqueness follows from Theorem~\ref{thm:uniqueness_eps}.
\end{proof}

To prove Theorem~\ref{thm:W_uniform_bounds}, we use the exact same comparison theorem as in~\cite{JB88}, relying on the maximum principle for linear parabolic PDE. We denote by $\bar B_R \subset \RR$ the closed ball centred at $0$ with radius $R>0$, and by $\Gamma_R:=\bar B_R \times \left\{0\right\} \cup \partial \bar B_R \times [0,T] $ the parabolic boundary of $Q_R := \bar B_R \times [0,T]$, whose interior we denote by $\mathring Q_R$.

\begin{lemma}[Maximum Principle, Friedman~\cite{Friedman64}] \label{lem:max_pple}
Define
\[
\mathscr{L} \varphi  := \partial_t \varphi  - \frac{\varepsilon}{2} \partial_{xx}^2 \varphi  + \partial_x \varphi  b^\varepsilon,
\]
where $b^\varepsilon$ is smooth. If $\mathscr{L} \varphi  \leq 0$ (respectively, $\geq 0$) in $\mathring Q_R$, then for all $(x,t) \in Q_R$,
\begin{align*}
 \varphi (x,t) \leq \sup_{(z,s)\in \Gamma_R} \varphi (z,s) \\
 \left( \text{respectively,} \qquad \inf_{(z,s)\in \Gamma_R} \varphi (z,s) \leq \varphi (x,t) \right).
\end{align*}
\end{lemma}
 
\begin{lemma}[Comparison theorem, James and Baras~\cite{JB88} Lemma 4.2] \label{lem:comparison_theorem}~%

Let $\varepsilon >0$. Let $\tilde w^\varepsilon$ be a solution of~\eqref{eq:w_epsilon_1d} %{eq:W_epsilon_d}
 over $\RR \times [0,T]$ and define
\[
\tilde{\mathscr{L}} : v \in C^1(\mathring Q_R; \RR) \mapsto \partial_t v - \frac{\varepsilon}{2} \partial_{xx}^2 v + g \partial_x v + \frac{1}{2} \left| \partial_x v \right|^2 - \tilde V_w^\varepsilon,
\]
$g(x) = {\rm{sgn}}(x) f(|x|)$ and $\tilde{V}_w^\varepsilon$ being defined in \eqref{eq:def_Vtildeeps}. Let $v \in C^1(\mathring Q_R; \RR)$ and $\varphi = v - \tilde w^\varepsilon$. If $\mathscr{L} v \geq 0$ (respectively, $\mathscr{L} v \leq 0$) in $\mathring Q_R$ and if $\tilde w^\varepsilon \leq v$ (resp. $v \leq \tilde w^\varepsilon$) on $\Gamma_R$, then $\tilde w^\varepsilon \leq v$ (resp. $v \leq \tilde w^\varepsilon$) in $\mathring Q_R$.
\end{lemma}

\begin{proof}[Same proof as in~\cite{JB88}:]~%
If $\tilde{\mathscr{L}} v \geq 0$, then subtract $\tilde{\mathscr{L}} w^\varepsilon = 0$ to get
\[
\partial_t \varphi - \frac{\varepsilon}{2} \partial_{xx}^2 \varphi  + g \partial_x \varphi + \frac{1}{2}\left(|\partial_x v|^{2}-\left|\partial_x \tilde w^{\varepsilon}\right|^{2}\right) \geq 0
\]
Now $|\partial_x v|^{2} - \left|\partial_x \tilde w^{\varepsilon}\right|^2 = \partial_x \varphi \cdot \left(\partial_x v+ \partial_x \tilde w^{\varepsilon}\right).$ Set
\[
b^{\varepsilon} = g + \frac{1}{2}\left(\partial_x v+ \partial_x \tilde w^\varepsilon \right) .
\]
Then $\mathscr{L} \varphi \geq 0$ and on $\Gamma_{R}$, $\varphi(z, s) \geq 0$. Hence $\varphi(x, t) \geq 0$ for all $(x, t) \in Q_{R}$ by Lemma~\ref{lem:max_pple}.
\end{proof}

\begin{proof}[Proof of Theorem~\ref{thm:W_uniform_bounds}]~%
The proof is very close to that given by James and Baras in~\cite{JB88}, inspired in~\cite{EI85}. It relies on the construction of a function $v$ independent of $\varepsilon$ such that $\tilde{\mathscr{L}} v \geq 0$ in $\mathring Q_R$ and $w^{\varepsilon} \leq v$ on $\Gamma_{R}$, independent of (sufficiently small) $\varepsilon>0$, which is achieved by making $v$ tend to $\infty$ close to the boundary. 

\paragraph{Proof of $(iv)$.}~%
Let $R \geq 4 \max (1, 2 \|g\|)$ and $\varepsilon \leq \varepsilon_R := \frac{1}{2 R^2}$. Define
\[
v(x, t)=\frac{1}{R^{2}-|x|^{2}}+\mu t+M
\]
where the constants $\mu>0, M>0$ will be adequately chosen later. Then %We denote by $C$ any constant -- so the value of $C$ may change from line to line. Then %write $v_{i}$ for $v_{x_{i}}$, etc. Then
\[
\begin{aligned}
\tilde{\mathscr{L}} v=& \mu-\frac{\varepsilon}{2}\left(\frac{2}{\left(R^{2}-|x|^{2}\right)^{2}}+\frac{8|x|^{2}}{\left(R^{2}-|x|^{2}\right)^{3}}\right)  + \frac{2 x}{\left(R^{2}-|x|^{2}\right)^{2}} \cdot g + \frac{2|x|^{2}}{\left(R^{2}-|x|^{2}\right)^{4}} - \tilde V_w^{\varepsilon} \\
= & \mu + \frac{1}{\left( R^2 - x^2 \right)^4} \left[ x^2  - \varepsilon \left(R^2 - x^2\right) \left(4x^2 + (R^2 - x^2) \right) \right]   + \mathscr{\tilde E}_R (x) + \mathscr{G}_R(x) - \tilde V_w^\varepsilon (x,t),
\end{aligned}
\]
where we define:
\begin{equation*}
\left\{\begin{aligned}
\mathscr{\tilde E}_R(x) & := \frac{2 x g (x)}{\left(R^2 - x^2 \right)^2} \quad \geq \quad - \mathscr{E}_R(x) := - \frac{2 |x| \|g\|_{L^\infty(\RR)}}{\left(R^2 - x^2 \right)^2}, \\
\mathscr{G}_R(x) & := \frac{x^2}{\left(R^2 - x^2 \right)^4} \geq 0.
\end{aligned}\right.\end{equation*}

Hence,
\[
\begin{aligned}
\tilde{\mathscr{L}} v \geq & \mu + \frac{1}{\left( R^2 - x^2 \right)^4} \left[ x^2  - \varepsilon \left(R^4 + 2 R^2 x^2 - 3 x^4 \right) \right] - \mathscr{E}_R (x) + \mathscr{G}_R(x) - \tilde V_w^\varepsilon (x,t) \\
\geq & \mu + \frac{1}{\left( R^2 - x^2 \right)^4} \left[ x^2  - \varepsilon \frac{4}{3 R^2} \right] - \mathscr{E}_R (x) + \mathscr{G}_R(x) - \tilde V_w^\varepsilon (x,t).
\end{aligned}
\]
For $\varepsilon < \varepsilon_R$: either $|x| \geq 1$ and it follows that $x^2  - \varepsilon \frac{4}{3 R^2} \geq 0$; or $|x| < 1$ and 
\[\frac{1}{\left( R^2 - x^2 \right)^4} \left[ x^2  - \varepsilon \frac{4}{3 R^2} \right] \geq \frac{-1}{\left(R^2 - 1\right)^4}.\]
Hence,
\begin{equation} \label{eq:Lv_fundamental}
\tilde{\mathscr{L}} v \geq \mu - \frac{1}{\left(R^2 - 1\right)^4} - \mathscr{E}_R (x) + \mathscr{G}_R(x) - \| \tilde V_w^\varepsilon\|_{L^\infty(\RR \times [0,T])}.
\end{equation}

\paragraph{Claim:}
For all $(x,t) \in Q_R$,
\[
- \mathscr{E}_R(x) + \mathscr{G}_R(x) \geq - 8 \max (1, \|g\|_{L^\infty(\RR)}^2).
\]

\begin{proof}[Proof of the Claim]
We will prove the Claim for $x \in [0, R)$ now. \textit{Mutatis mutandis}, the proof for $x \leq 0$ follows with no notable difference. Let $C = 4 \max \left(1, \sqrt {\|g\|} \right)$, and $\eta = \frac{1}{C \sqrt R}$.
\begin{itemize}
\item If $ x \leq R-\eta$:
\[
\begin{aligned}
\mathscr{E}_R (x) & \leq \mathscr{E}_R (R - \eta) = 2 \|g\| \frac{R-\eta}{\left( R^2 - \left(R-\eta\right)^2 \right)^2} = \frac{ C^2 \|g\| }{2}  \frac{R - \frac{1}{C \sqrt R}}{R - \frac{1}{C \sqrt R} + \frac{1}{4 C^2 R^2}} \\
	& \leq \frac{ C^2 \|g\| }{2} \leq 8 \max ( 1, \|g\|^2 ).
\end{aligned}
\]
And since $\mathscr{G}_R (x) \geq 0$, the claimed inequality is satisfied.
\item Otherwise, $ R - \eta < x < R$:
\[
(R^2-x^2)^2 \mathscr{G}_R (x) \geq (R^2-(R-\eta)^2)^2 \mathscr{G}_R (R-\eta) = \frac{\left(R-\eta \right)^2}{ \left(2\eta R - \eta^2 \right)^2} = \frac{C^2}{4} \frac{R^2 - \frac{2 \sqrt{R}}{C} + \frac{1}{C^2R}}{R - \frac{C}{\sqrt{R}} + \frac{C^2}{4R^2}}.
\]
To bound below the last fraction on the right-hand side, observe that since $C > 1$ and $R \geq 4$, we have $\frac{2R}{C} < \frac{R^2}{2}$; and since
\[
 R \sqrt{R} \geq 8 R \geq 8 \max (4, 8 \|g\|) > C = 4 \max (1, \sqrt{\|g\|}),
\]
  we have:
\[
\frac{C}{\sqrt{R}} - \frac{C^2}{4R^2} = \frac{C}{\sqrt{R}} \left[1 - \frac{C}{4R \sqrt{R}}\right] \geq 0.
\]
We obtain:
\[
(R^2-x^2)^2 \mathscr{G}_R (x) \geq \frac{C^2 R}{8}.
\]
Therefore,
\[
\mathscr{G}_R (x) - \mathscr{E}_R (x) = (R^2-x^2)^2 \left[ \mathscr{G}_R (x) - 2 x \|g\| \right] \geq  (R^2-x^2)^2 \left[ \frac{C^2 R}{8} - 2 R \|g\| \right] \geq 0,
\]
because $C^2 \geq 16 \|g\|$.
\end{itemize}
\end{proof}
From the Claim and equation~\eqref{eq:Lv_fundamental}, it is clear that $Lv \geq 0$ over $Q_R$, provided $\mu$ is chosen sufficiently large. Specifically,
\begin{equation} \label{eq:def_mu}
\mu = \frac{1}{\left(R^2 - 1\right)^4} + 8 (1 + \|g\|\|_{L^\infty(\RR)}) + \| \tilde V_w^\varepsilon\|_{L^\infty(\RR \times [0,T])}
\end{equation}
suffices. Choose now $M = \|w_0\|_{L^\infty(\RR)}$: large enough that
\[
w_{0}(x) \leq M \text { for all } x \in B_{R}.
\]
Since $v(x, t) \rightarrow \infty$ as $|x| \rightarrow R$ uniformly in $t \in[0, T]$, it follows from the maximum principle that
\[
\tilde w^{\varepsilon} \leq v \quad \text { in } \quad \mathring Q_R.
\]
Similarly, by considering $-v$ instead of $v$, we can find a similar upper bound for $\tilde w^\varepsilon$.

Since $v$ is continuous in $\mathring Q_R$ and $\max_{|x| \leq R/2} \frac{1}{\left(R^2-x^2\right)^2} = \frac{4}{3 R^2}$, the following bound over $Q_{R/2}$ follows.

\begin{equation} \label{eq:end_of_(i)}
\|\tilde w^\varepsilon \|_{L^\infty(Q_{R/2})} \leq \frac{4}{3 R^2} + \|w_0\|_{L^\infty} + \mu T,
\end{equation}
with $\mu$ defined in~\eqref{eq:def_mu}. Hence,
  \begin{equation} %\label{eq:sharp_Linfty}
\| \tilde w^\varepsilon \|_{L^\infty (Q_{R/2})} \leq \|\tilde w_0\|_{L^\infty (\RR)} + \left[8(1+\|g\|_{L^\infty (\RR)}) + \|V^\varepsilon\|_{L^\infty(\RR \times [0,T])} + \frac{1}{(R^2-1)^4} \right] T + \frac{4}{3 R^2}.
\end{equation}
The desired estimate follows, concluding the proof of $(iv)$. %We have proved the desired estimate over $Q_{R/2}$. since $R$ was arbitrarily set, this concludes the proof of $(iv)$.

\paragraph{Proof of $(i)$.}~%
$(iv) \Rightarrow (i)$.

\paragraph{Proof of $(ii)$.}~%
The estimate of the partial derivative in $x$ closely follows~\cite{JB88}, using a variant of the techniques in~\cite{EI85}. It consists of the following steps.

\begin{itemize}
%\item To simplify the notation we write $v=\tilde w^{\varepsilon}$ like James and Baras.
\item Define $Q \subset \subset Q^{\prime} \subset \subset \mathbb{R} \times (0, T)$, where $Q, Q^{\prime}$ are open and ``$\subset \subset$" means ``compactly contained in". 
\item Choose a smooth function $\zeta$ such that $\zeta \equiv 1$ on $Q$ and $\zeta \equiv 0$ near $\partial Q^{\prime}$, and define
\[
z:= \zeta^{2} |\tilde w^\varepsilon|^2 - \lambda \tilde w^\varepsilon,
\]
where $\lambda > 0$ will be chosen later.
\item Apply the maximum principle to $z$:
%\begin{itemize}
%\item 
$z$ reaches its maximum in $\bar Q'$. Assume it's reached at $(x_0, t_0) \in Q'$. Then, since $z$ is smooth,
\[
\left\{\begin{aligned}
\partial_x z & = 0, \\
0 & \leq \partial_t z - \frac{\varepsilon}{2} \partial_{x}^2 z. 
\end{aligned}\right.
\]
%where we denote $z_x := \partial_{x} z$, $z_{xx} = \partial_{x}^2 z$, etc.
\item Writing the previous inequality explicitly in terms of $\zeta$ and $\tilde w^\varepsilon$ and using the Hamilton-Jacobi equation satisfied by $\tilde w^\varepsilon$ yields, for $\varepsilon$ sufficiently small, at $(x_0, t_0)$:
\[
0 \leq - \partial_x \tilde w^\varepsilon \cdot \partial_x \left(\zeta^{2} |\partial_x \tilde w^\varepsilon|^2 \right) - g \cdot \partial_x \left(\zeta^{2} |\partial_x \tilde w^\varepsilon|^2\right) + \frac{\lambda}{2} |\partial_x \tilde w^\varepsilon|^2+ C \zeta|\partial_x \tilde w^\varepsilon|^{3}+C|\partial_x \tilde w^\varepsilon|^{2}+\lambda C|\partial_x \tilde w^\varepsilon|+\lambda C,
\]
%\[
%0 \leq - \partial_x \tilde w^\varepsilon \cdot \partial_x \left(\zeta^{2} |\partial_x \tilde w^\varepsilon|^2 \right) + \frac{\lambda}{2} |\partial_x \tilde w^\varepsilon|^2+ C \zeta|\partial_x \tilde w^\varepsilon|^{3}+C|\partial_x \tilde w^\varepsilon|^{2}+\lambda C|\partial_x \tilde w^\varepsilon|+\lambda C,
%\]
where we recall that $C$ is a generic constant name.
Using now $\partial_x z = 0$ at $(x_0, t_0)$, we have
\[
\frac{\lambda}{2}|\partial_x \tilde w^\varepsilon|^{2} \leq C \zeta|\partial_x \tilde w^\varepsilon|^{3}+C|\partial_x \tilde w^\varepsilon|^{2}+\lambda C|\partial_x \tilde w^\varepsilon|+\lambda C.
\]
\item Choosing $\lambda=\mu[(\max \zeta) |\partial_x w^\varepsilon|+1]$, with $\mu > 1$ to be chosen yields:
\[
\frac{\mu}{2}|\partial_x \tilde w^\varepsilon|^2 \leq C|\partial_x \tilde w^\varepsilon|^{2}+C \lambda \mu .
\]
%\item 
Hence for $\mu$ large enough, at $(x_0, t_0)$,
\[
|\partial_x \tilde w^\varepsilon|^{2} \leq C \lambda.
\]
%\item 
Hence:
\[
z \leq C \lambda \text { in } Q^{\prime} \text { . }
\]
\item If the max is reached at the boundary, the equation above holds since $\tilde w^\varepsilon$ is bounded. From it, James and Baras recover:
%\end{itemize}
\[
\max \zeta^{2}|\partial_x \tilde w^\varepsilon|^{2} \leq \max z+C \lambda \leq C \lambda
\]
and by definition of $\lambda$,
\[
\max \zeta^{2}|\partial_x \tilde w^\varepsilon|^{2} \leq C \mu[\max \zeta|\partial_x \tilde w^\varepsilon|+1],
\]
which implies
\[
\zeta|\partial_x \tilde w^\varepsilon| \leq C \quad \text { in } Q^{\prime} \text { , }
\]
so
\[
|\partial_x \tilde w^\varepsilon| \leq C \quad \text { in } \bar{Q},
\]
concluding the proof.~%

\paragraph{Proof of $(iii)$.}~%
Since $\hamilton$ is locally bounded, the conditions of~\cite[Lemma 5.2]{CL84} are met (with $\frac{\varepsilon}{2}$ here, instead of $\varepsilon$), which allows to conclude to the $\varepsilon$-dependent Hölder estimate:
\[
\forall (x,t), (x,s) \in Q, \qquad \left| \tilde w^\varepsilon (x,t) - \tilde w^\varepsilon (x,s) \right|  \leq K \left(\sqrt \varepsilon \left| t-s \right|^{1/2} + |t-s| \right).
\]
%Maximising the right-hand side over $\varepsilon \in (0,1)$ concludes the proof. 	
Since $\varepsilon \in (0,1)$, taking $\varepsilon = 1$ in the right-hand side concludes the proof.

%\paragraph{Proof of $(iv)$.}~%
%In the bound~\eqref{eq:end_of_(i)}, $\mu = (-\inf \hat \mu_R)^+$, with $\hat \mu_R$ defined in equation~\eqref{eq:hat_mu}. Note that choosing $-v$ as the test function to prove the corresponding lower bound in $(i)$ leads to the exact same condition on $\mu$, so no separate treatment is required. 

\end{itemize}
\end{proof}

\subsection{Viscosity solution limit -- Proof of Theorem~\ref{thm:S}.}

%\begin{remark}
%We could have chosen to prove a uniform-in-$\varepsilon$ Lipschitz-in-$t$ bound to complete the uniform boundedness in $W^{1,\infty}$. However, we choose to rather follow James and Baras in using Hölder boundedness in $t$ following~\cite[Lemma 5.2]{CL1984} to .
%\end{remark}

\begin{theorem} \label{thm:existence_and_uniqueness_HJ_w}
Assume $w_0$ is bounded and Lipschitz continuous, and $\hamilton$ satisfies the assumptions of Theorem \ref{thm:existence_and_uniqueness}.
Then there exists a unique viscosity solution of the limiting equation~\eqref{eq:w_limit}, defined over $(x,t) \in \RR_+ \times \RR_+$, and that solution can be obtained by the vanishing viscosity method.
\end{theorem}

\begin{proof}
The bounds $(i)$, $(ii)$, and $(iii)$ of Theorem~\ref{thm:w_restriction_uniform_bounds} and the Arzela-Ascoli theorem (see \textit{e.g.}~\cite[Theorem 1]{Lions:1985ve}) imply that there exists a decreasing subsequence $(\varepsilon_k)_{k \in \N}$ that tends to $0$ such that $w^{\varepsilon_k}$ converges uniformly over compact sets to a continuous function $w$. From bound $(iv)$, % Theorem~\ref{thm:w_restriction_uniform_bounds}.$(iv)$, 
it follows that $w$ is bounded over $\RR_+ \times [0,T]$. Since $\hamilton^\varepsilon$ also converges uniformly over compact sets to $\hamilton$, Proposition~\ref{prop:H_eps_CV_H_1d}, we may apply the stability result in~\cite{Barles93}. Uniqueness results from Theorem~\ref{thm:existence_and_uniqueness}.
\end{proof}

% Barles et al. 2013 contains other stability results for pde with some integral parts, but they're too much for us here.

%\todo{Check that everything works as expected with the Neumann boundary condition. Compactness does, but check the boundary condition in the viscosity sense.}

We may now consider, for $0 < \varepsilon < 1$ and for all $(x,t) \in \RR_+ \times [0,T]$:
\[
\entropy^\varepsilon (x,t) = w^\varepsilon (x,t) + y(t) h(x).
\]
By construction, $\entropy^\varepsilon$ is smooth and satisfies the second order evolution Hamilton-Jacobi equation~\eqref{eq:S_epsilon}. The uniqueness of the solution to that equation -- point $(i)$ of Theorem~\ref{thm:S} -- is a direct corollary of Theorem~\ref{thm:w_restriction_uniform_bounds}. The local uniform convergence of the Hamiltonian $\hamilton_{\entropy}^\varepsilon$ to $\hamilton_{\entropy}$, point $(ii)$, results from Proposition~\ref{prop:H_eps_CV_H_1d}. Since $y$ and $h$ are bounded and have bounded derivatives, appropriate bounds can be obtained on $\entropy^\varepsilon$ of the type of those in Theorem~\ref{thm:w_restriction_uniform_bounds}. So $(iii)$, the convergence of $\entropy^\varepsilon$ to $\entropy$, follows from that of $w^\varepsilon$ to $w$ in the proof of Theorem~\ref{thm:existence_and_uniqueness_HJ_w}.

As for point $(iv)$, the well-posedness of the limit equation follows from that of equation~\eqref{eq:w_limit}: Theorem~\ref{thm:existence_and_uniqueness_HJ_w}. Since $y$ and $h$ are smooth enough, $yh$ may be added or subtracted to any test function, guaranteeing that definition~\ref{def:visc_sol_HJ} applies for $w$ in equation~\eqref{eq:w_limit} if and only if it applies for $\entropy$ in equation~\eqref{eq:S}. The vanishing viscosity limit procedure also works in a similar way, concluding the proof of Theorem~\ref{thm:S}.

\section{Dynamic programming principle for the HJB limit}

At this point, we have identified a functional from the viscosity solution limit which could be a good candidate for defining the cost-to-come associated with the  Mortensen estimator  of the Skorohod’s problem. Unfortunately, despite its stochastic interpretation, the viscosity solution limit cannot be linked as in \cite{JB88} to the cost-to-come of the deterministic filtering problem.

\subsection{A control problem interpretation of the limit solution}\label{ssec:control_interpretation} 

The limit $w(x,t) = \entropy(x,t) - y(t)h(x)$ can be characterised as the unique viscosity solution of the HJB equation \eqref{eq:w_limit}. Following the method of \cite{JB88}, a backward control process is now built whose cost function $\costshift$ will be identified to $w$. Consider the control process associated to the $\R_{+}$-valued backward trajectories $\left( z_{\omega}^{x,t}(s) \right)_{0 \leq s \leq t}$ defined by
%{\color{orange}
\begin{equation}  \label{controlBack}
\begin{cases}
%\dot{z}_{\omega}^{x,t}(s) \in 
%\omega (s) + \partial \mathcal{I}_{\RR_+} \left( z_\omega^{x,t}(s) \right), & 0 \leq s \leq t,  \\
% \forall s \mbox{ a.e.} \in [0,t], \: \forall q \geq 0, \: (\dot z_\omega^{x,t}(t)-\omega(t))(q - z_\omega^{x,t}(t)) \leq 0\\
 \forall s \mbox{ a.e.} \in [0,t], \: \forall q \geq 0, \: (\dot z_\omega^{x,t}(s)-\omega(s))(q - z_\omega^{x,t}(s)) \leq 0\\
z_{\omega}^{x,t}(t) = x, 
\end{cases}
\end{equation}
%$\mathcal{I}_{\RR_+} = \infty \mathbf{1}_{(-\infty,0)}$ still being the convex characteristic function of $[0,\infty)$. 
%}
Assume this system is partially known through the perturbed observation function $y(t)$ given by
\[ 
	\dot{y}(s) = h( z_{\omega}^{x,t}(s) ) + \eta (s),
\]
the control parameters $\eta$ and $\omega$ being square-integrable $\R$-valued functions of time. To each such $\omega$ can be associated a backward trajectory $z_{\omega}^{x,t}$. The control problem then consists in minimizing a functional $\psi( z_{\omega}^{x,t}(0) )$ of the arrival point at time $0$, together with the $L^2$ weights of control functions $\omega$ and $\eta$. The cost rate is thus
\[ 
	\tilde{\timeError}(z,\omega,s) := \frac{1}{2} \omega^2 + \frac{1}{2} | \dot{y}(s) - h ( z ) |^2,
\]
so that the cost to go (backward in time) from $x$ at time $t$ to time $0$ reads
\[ 
	\inf_{\omega \in L^2(0,t)} \psi ( z_{\omega}^{x,t}(0) ) + \int_0^t \tilde{\timeError} \left( z_{\omega}^{x,t}(s),{\omega}(s),s \right) \dd s, 
\]
developing the square $| \dot{y}(s) - h ( z(s) ) |^2$, the term $| \dot{y}(s) |^2$ doesn't affect the minimization problem, and the cost rate can be chosen to be
\[ 
	\timeError(z(s),{\omega}(s),s) := \frac{1}{2}{\omega}^2(s) +  \frac{1}{2} h^2 ( z(s) ) - \dot{y}(s) h ( z(s) ),
\]
as required to take the limit in the probabilistic setting. This leads to the functional
\[ 
	\criter(x,\omega,t) := \psi( z_{\omega}^{x,t}(0) ) + \int_0^t \timeError \left( z_{\omega}^{x,t}(s),{\omega}(s),s \right) \dd s,
\]
and the cost function $\costshift(x,t) := \inf_{\omega} \criter(x,\omega,t)$ will appear to be the desired target function. Note the initial value condition $\costshift(x,0) = \psi(x)$.

\begin{lemma}[Principle of Optimality] \label{Optiprinciple}
Consider a terminal point $(x,t)$ together with a control ${\omega}$; then for every $0 < \tau < t$
\[ \costshift(x,t) = \inf_{\omega \in L^2(t-\tau,t)} \left[ \costshift \left( z^{x,t}_{\omega} ( t - \tau ) , t-\tau \right) + \int_{t-\tau}^t \timeError \left( z_{\omega}^{x,t}(s),{\omega}(s),s \right) \dd s \right].
\]
\end{lemma}

\begin{proof} Given another control $\left( {\omega}'(s) \right)_{0 \leq s \leq t-\tau}$, define the square-integrable control
\begin{equation*}  
\tilde{\omega}(s) = 
\begin{cases}
{\omega}'(s) &\text{  if  } 0 \leq s < t - \tau,\\
{\omega}(s) & \text{ if  } t - \tau \leq s \leq t. 
\end{cases}
\end{equation*}
For $s < t- \tau$ note that  $z_{\tilde{\omega}}^{x,t}(s) = z^{t-\tau,z^{x,t}_{\omega} ( t - \tau )}_{{\omega}'} (s )$, so that by definition of $\costshift$
\[ \costshift(x,t) \leq \psi \left( z^{t-\tau,z^{x,t}_{\omega} ( t - \tau )}_{{\omega}'} ( 0 ) \right) + \int_0^{t-\tau} \timeError \left( z^{t-\tau,z^{x,t}_{\omega} ( t - \tau )}_{{\omega}'} (s ) ,{\omega}'(s),s \right) \dd s + \int_{t-\tau}^t \timeError \left( z_{\omega}^{x,t}(s),{\omega}(s),s \right) \dd s, 
\]
and taking the infimum over ${\omega}'$ concludes. Equality is achieved by considering a sequence of controls whose costs converge towards the infimum.
\end{proof}

\begin{lemma}[Uniform terminal continuity] \label{UnifCont}
Consider a terminal point $(x,t)$ and $M > 0$; then $s \mapsto z^{x,t}_{\omega} ( s )$ is continuous at the terminal point $s = t$ uniformly in ${\omega}$ such that $\criter(x,\omega,t) \leq M$.
\end{lemma}

\begin{proof} Consider $\varepsilon > 0$ and a control ${\omega}$. If $x > 0$, the continuity of $z^{x,t}_{\omega}$ at $t$ guarantees that
\[ 
	\tau_{\omega} := \sup \left\{ \tau > 0 \, , \, z^{x,t}_{\omega}(t-\tau) > 0 \text{ and }  | z^{x,t}_{\omega}(t-\tau) - x | \leq \varepsilon \right\} > 0.
\]
Considering $0 < \tau < \min(\tau_{\omega},1)$ to make sure that $z^{x,t}_{\omega}(s) > 0$, one has $\dot{z}^{x,t}_{\omega}(s) = {\omega}(s)$ for $t - \tau \leq s \leq t$ thanks to \ref{controlBack}. Thus %By continuity, the function $b$ is bounded by a positive constant $K$ on $\left[x-\varepsilon,x+\varepsilon\right] \cap \R_{\geq 0}$, so that
\[ x - z^{x,t}_{\omega}(t-\tau) = \int_{t-\tau}^t \dot{z}^{x,t}_{\omega}(s) \dd s = \int_{t-\tau}^t {\omega}(s) \dd s. 
\]
Using Cauchy-Schwarz inequality
\[ 
	| x - z^{x,t}_{\omega}(t-\tau) | \leq \sqrt{2\tau \criter(x,\omega,t)} \leq \sqrt{2 \tau M}, 
\]
and this proves the bound $\tau_{\omega} \geq \left( \frac{\min(\varepsilon,|x|)}{\sqrt{2M}} \right)^2$, the right-hand side being independent of ${\omega}$. \\
In the case $x =0$, consider
\[ 
	\tau^0_{\omega} := \sup \left\{ \tau > 0 \, , \, \forall t-\tau \leq s \leq t, \, z^{x,t}_{\omega}(s) = 0 \right\},
\]
\[ 
	\tau^1_{\omega} := \sup \left\{ \tau > \tau^0_{\omega} \, , \, z^{x,t}_{\omega}(t-\tau) > 0 \text{ and } z^{x,t}_{\omega}(t-\tau) \leq \varepsilon \right\}.
\]
The continuity of $z^{x,t}_{\omega}$ indeed guarantees $\tau^0_{\omega} < \tau^1_{\omega}$; since for $\tau_{\omega}^0< \tau < \tau' < \tau^1_{\omega}$ 
\[ z^{x,t}_{\omega}(t-\tau) - z^{x,t}_{\omega}(t-\tau') = \int_{t-\tau'}^{t-\tau} \dot{z}^{x,t}_{\omega}(s) \dd s =\int_{t-\tau'}^{t-\tau} {\omega}(s) \dd s,
\]
the same reasoning as above gives a positive lower bound for $\tau_{\omega}^1 - \tau_{\omega}^0$ which is independent of ${\omega}$, completing the proof.
\end{proof}

The function $\costshift$ can now be identified to the previous limit using the HJB equation \eqref{eq:w_limit}. Note that the Hamiltonian $\hamilton$ can equivalently be defined as 

\begin{equation} \label{Hamiltonian}
\hamilton(x,t,\lambda) = \max_{{\omega}' \in \R}  \lambda {\omega}' - \timeError \left( x,{\omega}',t \right).
\end{equation}

\begin{proposition}[Sub-solution] \label{prop:sub}
The function $\costshift$ is a viscosity sub-solution of \eqref{eq:w_limit}. 
\end{proposition}

\begin{proof} For $x \geq 0$ and $t >0$, consider a $C^1$ test function $\phi$ such that $\costshift-\phi$ has a local maximum at point $(x,t)$. For any control ${\omega}$ and every $\tau > 0$ small enough, this leads to 
\[ \costshift\left( z^{x,t}_{\omega} ( t - \tau ) , t-\tau \right) - \phi \left( z^{x,t}_{\omega} ( t - \tau ) , t-\tau \right) \leq \costshift(x,t) - \phi(x,t),
\]
because of $z^{x,t}_{\omega}(t) = x$ and the continuity of $z^{x,t}_{\omega}$ at $(x,t)$. Therefore, we have 
\[
	\phi(x,t) - \phi \left(  z^{x,t}_{\omega} ( t - \tau ) , t-\tau \right) \leq \costshift(x,t) - \costshift \left( z^{x,t}_{\omega} ( t - \tau ) , t-\tau \right) \leq  \int_{t-\tau}^t \timeError \left( z_{\omega}^{x,t}(s),{\omega}(s),s \right) \dd s,
\]
using the principle of optimality \ref{Optiprinciple}. Dividing by $\tau$ and taking the $\tau \rightarrow 0^+$ limit gives 
\[ \left.\frac{\dd}{\dd s}\right|_{s=t} \phi \left( z^{x,t}_{\omega} ( s ),s \right) \leq \timeError \left( z_{\omega}^{x,t}(t),{\omega}(t),t\right),
\]
so that
\[ 
	\partial_t \phi(x,t) + \dot{z}^{x,t}_{\omega} ( t ) \partial_x \phi(x,t) - \timeError \left( x ,{\omega}(t),t \right) \leq 0.
\]
Then, we have
\[ \partial_t \phi(x,t) + {\omega}(t) \partial_x \phi(x,t) - \timeError \left( x ,{\omega}(t),t\right) \leq \partial_x \phi(x,t) \left[ {\omega}(t) - \dot{z}^{x,t}_{\omega} ( t ) \right].
\]
If $x > 0$ then $\dot{z}^{x,t}_{\omega} ( t ) = {\omega}(t)$ according to \eqref{controlBack}; 
else $x =0$ so that ${\omega}(t) - \dot{z}^{x,t}_{\omega} ( t ) \geq 0$, and one can assume $\partial_x \phi(0,t) \leq 0$ following the definition \eqref{def:visc_sol_HJ}. In every case
\[  
	\partial_x \phi(x,t) \left[ {\omega}(t) - \dot{z}^{x,t}_{\omega} ( t ) \right] \leq 0. 
\]
Since this is true for every ${\omega}$, taking the maximum over ${\omega}(t)$ allows to recover \eqref{Hamiltonian} and
\[ 
	\partial_t \phi(x,t) + \hamilton \left( x ,t,\partial_x \phi(x,t) \right) \leq 0,
\]
as desired.
\end{proof}

\begin{proposition}[Super-solution] \label{prop:sur}
The function $\costshift$ is a viscosity super-solution of \eqref{eq:w_limit}.
\end{proposition}

\begin{proof}
For $x \geq 0$ and $t >0$,  consider a $C^1$ test function $\phi$ such that $\costshift- \phi$ has a local minimum at point $(x,t)$. Positive numbers $\delta,\delta'>0$ exist such that 
\begin{equation} \label{MinVisc}
 |t-t'| \leq \delta \text{ and } |x-x'| \leq h \Rightarrow \costshift(x',t') - \phi(t',x') \geq \costshift \left( x,t \right) - \phi \left( x,t \right).
\end{equation}    
Fix now $\varepsilon > 0$ and $M > \costshift(x,t)$. By lemma \ref{UnifCont} $\delta' > 0$ exists such that for every ${\omega}$ with $\criter(x,\omega,t) \leq M$ 
\[  
	0 \leq \tau \leq \delta' \Rightarrow | z_{\omega}^{x,t}(t-\tau) - x | \leq h.
\]
Consider a sequence $(\tau_n)_{n \geq 0}$ which converges to $0$ with $0 < \tau_n \leq \min ( \delta , \delta') $. In the principle of optimality \ref{Optiprinciple} which characterises $\costshift(x,t)$, it is sufficient to minimize over $\omega$ with $\criter(x,\omega,t) \leq M$, because $M > \costshift(x,t)$. Then by definition of the infimum, ${\omega}_n$ with $\criter(x,\omega_n,t) \leq M$ exists for every $n$, satisfying
\[ 
	\costshift(x,t) +  \varepsilon \tau_n \geq \costshift \left( z^{x,t}_{{\omega}_n} (t-\tau_n), t-\tau_n \right) + \int_{t-\tau_n}^t \timeError \left( z_{{\omega}_n}^{x,t}(s),{\omega}_n(s),s \right) \dd s.
\]
Using \ref{MinVisc}, it follows
\begin{align*} 
\phi \left( x,t \right) - \phi \left( z_{{\omega}_n}^{x,t}(t-\tau_n),t-\tau_n \right) &\geq \costshift \left( x,t \right) - \costshift \left( z_{{\omega}_n}^{x,t}(t-\tau_n),t-\tau_n \right) \\
&\geq - \varepsilon \tau_n + \int_{t-\tau_n}^t \timeError \left( z_{{\omega}_n}^{x,t}(s),{\omega}_n(s),s \right) \dd s. 
\end{align*}
The functions $\phi$ and $z_{{\omega}_n}^{x,t}$ being differentiable, taking the $s$-derivative in $\phi \left( z^{x,t}_{\omega} ( s ) , s \right)$ leads
\[ 
	\phi \left( x,t \right) - \phi \left( z_{{\omega}_n}^{x,t}(t-\tau_n) ,t-\tau_n \right) = \int_{t-\tau_n}^t \partial_t \phi \left( z^{x,t}_{{\omega}_n} (s),s \right) + \dot{z}^{x,t}_{{\omega}_n} (s) \partial_x \phi \left(z^{x,t}_{{\omega}_n} (s),s \right) \dd s.  
\]
Therefore, we have
\[ \int_{t-\tau_n}^t \partial_t \phi \left( z^{x,t}_{{\omega}_n} (s),s \right) + \dot{z}^{x,t}_{{\omega}_n} (s) \partial_x \phi \left( z^{x,t}_{{\omega}_n} (s),s \right) - \timeError \left( z_{{\omega}_n}^{x,t}(s),{\omega}_n(s),s \right) \dd s \geq - \varepsilon \tau_n. 
\]
Adding $\int_{t-\tau_n}^t \partial_x \phi \left( z^{x,t}_{{\omega}_n} (s) ,s\right) {\omega}_n (s) \dd s$ to each side,
\begin{align*} 
\int_{t-\tau_n}^t \partial_t \phi \left( z^{x,t}_{{\omega}_n} (s) ,s\right) &+ \partial_x \phi \left( z^{x,t}_{{\omega}_n} (s),s \right) {\omega}_n(s) - \timeError \left( z_{{\omega}_n}^{x,t}(s),{\omega}_n(s),s \right) \dd s \\
&\geq - \varepsilon \tau_n + \int_{t-\tau_n}^t \partial_x \phi \left( z^{x,t}_{{\omega}_n} (s),s \right) \left[ {\omega}_n (s) - \dot{z}^{x,t}_{{\omega}_n} (s) \right] \dd s. 
\end{align*}
Note now that 
\begin{align*} 
\hamilton \left( z^{x,t}_{{\omega}_n} (s),s, \partial_x \phi \left( z^{x,t}_{{\omega}_n} (s) ,s\right) \right) &= \max_{{\omega}' \in \R} \partial_x \phi \left( z^{x,t}_{{\omega}_n} (s),s \right) {\omega}'- \timeError \left( z_{{\omega}_n}^{x,t}(s),{\omega}',s \right) \\
&\geq \partial_x \phi \left( z^{x,t}_{{\omega}_n} (s),s \right) {\omega}_n(s) - \timeError \left( z_{{\omega}_n}^{x,t}(s),{\omega}_n(s),s \right).
\end{align*} 
Moreover if $x > 0$, the uniform convergence of lemma \ref{UnifCont} allows to take $n$ large enough so that $z_{{\omega}_n}^{x,t}(s) > 0$ for $t - \tau_n \leq s \leq t$ and thus $\dot{z}^{x,t}_{{\omega}_n} (s) = {\omega}_n (s)$. If $x = 0$ one can assume $\partial_x \phi \left( 0, t \right) \geq 0$, and use the fact that $\omega_n(s) - \dot{z}^{x,t}_{{\omega}_n} (s) \geq 0$ by \ref{controlBack}, with equality when $z^{x,t}_{{\omega}_n} (s) > 0$. In every case 
\[ 
	\int_{t-\tau_n}^t \partial_x \phi \left( z^{x,t}_{{\omega}_n} (s),s \right) \left[ {\omega}_n (s) - \dot{z}^{x,t}_{{\omega}_n} (s) \right] \dd s  \geq 0, 
\]
for $n$ large enough. Thus
\[ 
	\int_{t-\tau_n}^t \partial_t \phi \left( z^{x,t}_{{\omega}_n} (s),s \right) + \hamilton \left( z^{x,t}_{{\omega}_n} (s),s, \partial_x \phi \left( z^{x,t}_{{\omega}_n} (s),s \right) \right) \dd s \geq - \varepsilon \tau_n.
\]
Lemma \ref{UnifCont} guarantees the continuity of $s \mapsto z^{x,t}_{{\omega}_n} (s)$ at $s = t$ uniformly in ${\omega}_n$ such that $W_{{\omega}_n} (t,s) \leq M$, so that dividing by $\tau_n$ and taking the $n \rightarrow + \infty$ limit gives
\[ \partial_t \phi \left( t, x \right) + \hamilton \left( x,s, \partial_x \phi \left( x,t) \right) \right) \geq - \varepsilon.
\]
Since this hold for every $\varepsilon >0$, this concludes the proof.
\end{proof}

\begin{theorem}[Identification]
Using the uniqueness result \ref{thm:uniqueness_eps}, it is now possible to identify the solution $w$ of \eqref{eq:w_limit} to $\costshift$, provided that the initial condition is $\psi(x) = w_0(x)$.
\end{theorem}

\noindent This establishes the desired link between the stochastic filtering problem \eqref{svi} and the control problem  \eqref{controlBack}. In particular, the limit doesn't allow to compute a recursive estimator, because it stems from a control problem and not a filtering one. The estimation has thus to be done by keeping some approximating noise with (small) amplitude $\varepsilon >0$, or using the penalised dynamics.

\subsection{Lost equivalence with the Mortensen estimator}

Let's go back to the estimation problem of the constrained dynamics \eqref{eq:dynsys} with $f=0$, namely the Skorokhod problem:
\begin{equation}\label{eq:dynsysf0}
\begin{cases}
 \forall t \mbox{ a.e.} \in [0,T], \: \forall z \geq 0, \: (\omega(t) - \dot x(t))(z- x(t)) \leq 0\\ 
\state(0) = \initNoise.
\end{cases}
\end{equation}

As in Section \ref{sec:penal-moretensen}, it could be tempting to use a direct deterministic filtering approach base on the cost to come
\begin{equation*}
\costcome(x,t) := \inf_{(\initNoise, \modelNoise) \in \mathcal{A}_{x, t}} \left [ \initError(\initNoise) + \int_0^t \timeError(\state_{|\initNoise,\modelNoise}(s),\modelNoise(s),s)\, \diff s \right ],
\end{equation*}
where we omit $\dot{y}$ to simplify the notation and the pre-image set can be also defined by
\[
\mathcal{A}_{x, t} := \left\{\left(\initNoise, \modelNoise\right) \in \R^+ \times \Ltwo[0,t]  : \state_{|\initNoise,\modelNoise} \text{ follows } \eqref{eq:dynsysf0} \text{ with } \state_{|\initNoise,\modelNoise}(0)=\initNoise, \state_{|\initNoise,\modelNoise}(t) = \state \right\}.
\]
This admissible set is never empty, because it is always possible to reach every $x \geq 0$ at time $t$ starting from any positive $\initNoise > 0$ by considering a (slow enough) straight line without reflection. However, the dynamics \eqref{eq:dynsys} is now well-posed in forward time only: given a value $x$ at time $t>0$ and a control $\omega$, there's no more well-posedness for the backward in time problem starting from $x$ at time $t$. This feature is due to the non-reversibility introduced by the reflection and complicates the situation a lot, because $\costcome(x,t)$ can no more be easily computed. In particular, it is no more the solution of the expected HJB equation \eqref{eq:w_limit}. Indeed, let's try to show --~as done for $\mathcal{W}$ in Proposition \ref{prop:sub}~-- that $\costcome$ is a viscosity sub-solution of equation \eqref{eq:w_limit}. First of all, one could prove the analogous of Theorem \ref{PrincipleBell}, which would read here:
\begin{equation}\label{eq:OptiprincipleConstraint}
\costcome(x,t) = \inf_{(\zeta,\omega) \in \mathcal{A}_{x, t}} \left[ \costcome \left( \state_{|\initNoise,\modelNoise} ( t - \tau ) , t-\tau \right) + \int_{t-\tau}^t \timeError \left( \state_{|\initNoise,\modelNoise} (s ),{\omega}(s),s \right) \dd s \right],	
\end{equation}

\noindent Let's now mimic the proof of Proposition \ref{prop:sub}: for $x \geq 0$ and $t >0$, consider a $C^1$ test function $\phi$ such that $\costcome- \phi$ has a local maximum at point $(x,t)$. For any control ${\omega}$, any initial condition $\initNoise$ and every $\tau > 0$ small enough, this leads to 
\[ \costcome \left( \state_{|\initNoise,\modelNoise} ( t - \tau ) , t-\tau \right) - \phi \left( \state_{|\initNoise,\modelNoise} ( t - \tau ) , t-\tau \right) \leq \costcome(x,t) - \phi(x,t), 
\]
because of $\state_{|\initNoise,\modelNoise}(t) = x$ and the continuity of $\state_{|\initNoise,\modelNoise}$ at $(x,t)$. We have 
\[\phi(x,t) - \phi \left(  \state_{|\initNoise,\modelNoise} ( t - \tau ) , t-\tau \right) \leq \costcome(x,t) - \costcome\left( \state_{|\initNoise,\modelNoise} ( t - \tau ), t-\tau \right) \leq  \int_{t-\tau}^t \timeError \left( \state_{|\initNoise,\modelNoise}(s),{\omega}(s),s \right) \dd s, 
\]
using the principle of optimality given by \eqref{eq:OptiprincipleConstraint}. Dividing by $\tau$ and taking the $\tau \rightarrow 0^+$ limit gives
\[ \left.\frac{\dd}{\dd s}\right|_{s=t} \phi \left( \state_{|\initNoise,\modelNoise} ( s ),s \right) \leq \timeError \left( \state_{|\initNoise,\modelNoise}(t),{\omega}(t),t\right), 
\]
so that
\[ \partial_t \phi(x,t) + \dot{\state}_{|\initNoise,\modelNoise} ( t ) \partial_x \phi(x,t) - \timeError \left( x ,{\omega}(t),t \right) \leq 0.
\]
Then
\begin{equation}\label{eq:wrong-bnoundary}
\partial_t \phi(x,t) + {\omega}(t) \partial_x \phi(x,t) - \timeError \left( x ,{\omega}(t),t\right) \leq \partial_x \phi(x,t) \left[ {\omega}(t) - 
\dot{\state}_{|\initNoise,\modelNoise}( t ) \right]. 
\end{equation}
If $x > 0$, then $\dot{\state}_{|\initNoise,\modelNoise} ( t ) = \omega(t)$ according to \eqref{eq:dynsysf0}. Hence
\[ 
\partial_t \phi(x,t) + \hamilton \left( x ,t,\partial_x \phi(x,t) \right) =0 \leq 0, 
\]
as desired. However if $x =0$ then ${\omega}(t) - \dot{\state}_{\penal|\initNoise,\modelNoise}( t ) \leq 0$ by definition of the sub-differential dynamics \eqref{eq:dynsysf0}. Considering $\phi$ such that  $\partial_x \phi(0,t) \leq 0$, we get
\[  
\partial_x \phi(x,t) \left[ {\omega}(t) - \dot{\state}_{|\initNoise,\modelNoise}( t ) \right] \geq 0,
\]
which, when combined to \eqref{eq:wrong-bnoundary} does not allow to constrain $\partial_t \phi(x,t) + {\omega}(t) \partial_x \phi(x,t) - \timeError \left( x ,{\omega}(t),t\right)$ to be non-positive.  The boundary condition  appears to be
\begin{equation}\label{eq:bound-wrong-sign}
    \min \left\{  + \partial_x \phi (0,t_0) \; , \;  (\partial_t \phi + H(\cdot, u, \partial_x \phi))(0,t_0) \right\} \leq 0.
\end{equation}
A similar situation would arise if one tried to prove the super-solution property for $\costcome$ (the analog of Proposition \ref{prop:sur}). We therefore believe that the connection between the viscosity limit of stochastic filtering and deterministic filtering for dynamics nonreversible in time is broken, and $\costcome$ cannot be computed from a forward dynamics that appears --~from \eqref{eq:bound-wrong-sign}~-- to be an ill-posed HJB dynamics.

As a result, a recursive estimator of \eqref{eq:dynsysf0} --~and similarly for \eqref{eq:dynsys}~-- cannot be the Mortensen estimator computed from $\costcome$ which does not appear to follow a well-posed Hamilton-Jacobi-Bellman equation. As a consequence, to obtain a computable sequential estimator, one must choose between two alternatives: 
\begin{itemize}
\item[-] approximate the dynamics \eqref{eq:dynsys} with the penalised dynamics \eqref{eq:pen_det}, resulting in an approximate Moretensen estimator; 
\item[-] define the stochastic filtering problem in terms of \eqref{svi} and use the tools of stochastic filtering and particle filtering \cite{Chen:2003up} for a small  but nonzero value of $\varepsilon$.
\end{itemize}

%%%%%%%%%%%%%%%%%%%%%%%%%%%%
%%%%%%%%%%%%%%%%%%%%%%%%%%%%
\section{Appendix} \label{sec:Appendix}
%%%%%%%%%%%%%%%%%%%%%%%%%%%%
%%%%%%%%%%%%%%%%%%%%%%%%%%%%

\subsection{Derivation of the robust Zakai equation} \label{zakai}

\begin{lemma}[Robust Zakai equation]
The random function $p^{\varepsilon}$ satisfies the robust Zakai equation, adding some Robin boundary conditions:
\begin{equation} \label{eq:RobZakBound}
\begin{cases}
\partial_t p^{\varepsilon} (t,\cdot) -y(t) h'(x) \partial_x p^{\varepsilon} (t,\cdot) + \frac{1}{\varepsilon} \potential^{\varepsilon}(x,t) p^{\varepsilon} (t,\cdot) = \dfrac{\varepsilon}{2} \partial^2_{xx} p^{\varepsilon}(t,\cdot),& (x,t)\in\RR^+ \times \RR^+\\ 
\dfrac{\varepsilon}{2} \partial_{x} p^{\varepsilon} (0,t) + \dfrac{Y_t h'(x)}{2} p^{\varepsilon} (0,t) = 0, &t\in\RR^+, 
\end{cases}
\end{equation}
where
\[ \potential^{\varepsilon}(x,t) := \frac{h^2(x)}{2} - \frac{\varepsilon}{2} Y_t h''(x) - \frac{1}{2} Y^2_t ( h'(x))^2. \]
\end{lemma}

This recovers a result in \cite{davis1980multiplicative} for robust filtering of reflected diffusion.

\begin{proof}
Thanks to Girsanov change of measure (see e.g. \cite{zakai1969optimal,bain2008fundamentals}), it is sufficient to treat the case where $\left( \frac{Y_t}{\sqrt{\varepsilon}} \right)_{t \geq 0}$ is a standard brownian motion. Then, using Ito's rule for stochastic differential calculus
\[ \dd p^{\varepsilon} (x,t) = \exp \left[ -\frac{Y_t h(x)}{\varepsilon} \right] \dd q^{\varepsilon}(x,t) + q^{\varepsilon}(x,t) \dd \exp \left[ -\frac{Y_t h(x)}{\varepsilon} \right] + \dd \left[ \exp \left[ -\frac{Y_{\cdot} h(x)}{\varepsilon} \right] , q^{\varepsilon} (x,\cdot) \right]_t ,\]
the quadratic cross-variation being given by 
\[ \dd \left[ \exp \left[ -\frac{Y_{\cdot} h(x)}{\varepsilon} \right] , q^{\varepsilon} (x,\cdot) \right]_t = - \frac{q^{\varepsilon}(x,t) h^2(x)}{\varepsilon} \exp \left[ -\frac{Y_{t} h(x)}{\varepsilon} \right]. \]
Moreover, by Ito's rule
\[  \dd \exp \left[ -\frac{Y_t h(x)}{\varepsilon} \right] = - \frac{h(x)}{\varepsilon} \exp \left[ -\frac{Y_t h(x)}{\varepsilon} \right] \dd Y_t + \frac{h^2(x)}{2 \varepsilon} \exp \left[ -\frac{Y_t h(x)}{\varepsilon} \right] \dd t, \]
using \eqref{ZakaiBiundary}
\[ \dd q^{\varepsilon}(x,t) = \frac{\varepsilon}{2} \partial^2_{xx} q^{\varepsilon}(x,t) \dd t + \frac{ q^{\varepsilon}(x,t) }{\varepsilon} \dd Y_t, \]
this gives
\[ \frac{\dd}{\dd t} p^{\varepsilon} (x,t) = \frac{\varepsilon}{2} \exp \left[ -\frac{Y_t h(x)}{\varepsilon} \right] \partial^2_{xx} p^{\varepsilon}(x,t) - \frac{h^2(x)}{2 \varepsilon} p^{\varepsilon} (x,t), \]
noticing that
\[ \exp \left[ -\frac{Y_t h(x)}{\varepsilon} \right] \partial_x q^{\varepsilon} (x,t) = \partial_x p^{\varepsilon} (x,t) + \frac{Y_t h'(x)}{\varepsilon} p^{\varepsilon} (x,t), \]
it is straightforward to obtain that
\begin{align*} 
\partial^2_{xx} p^{\varepsilon} (x,t) = \exp \left[ -\frac{Y_t h(x)}{\varepsilon} \right] \partial^2_{xx} q^{\varepsilon} (x,t) &- \frac{2 Y_t h'(x)}{\varepsilon} p^{\varepsilon} (x,t) \\
&- p^{\varepsilon} (x,t) \left[ \frac{\left( Y_t h'(x) \right)^2}{\varepsilon} + \frac{Y_t h''(x)}{\varepsilon}. \right]  
\end{align*}
Gathering everything
\[
\frac{\dd}{\dd t} p^{\varepsilon} (x,t) = Y_t h'(x) \partial_{x} p^{\varepsilon}(x,t) + \frac{p^{\varepsilon}(x,t)}{\varepsilon} \left( -\frac{h^2(x)}{2} + \frac{\left( Y_t h'(x) \right)^2}{2} + \frac{\varepsilon}{2} Y_t h''(x) \right) + \frac{\varepsilon}{2} \partial^2_{xx} p^{\varepsilon}(x,t) ,
\]
which is the desired equation. The boundary conditions are directly obtained from the ones in \eqref{ZakaiBiundary}.
\end{proof} 

In equation \eqref{eq:RobZakBound}, note that the random variable $Y_t$ just behaves as a parameter, which only appears inside the coefficients. This parameter $Y_t$ being defined as the function $\omega \in \Omega \mapsto Y(t,\omega)$, this can be seen as a family of deterministic PDEs indexed by a parameter $\omega$. At this point, it is only necessary to consider given realisations of the trajectory, i.e. continuous deterministic functions $(y(s))_{0\leq s \leq t}$. The remaining question will then be the measurability of the solution in $\omega$, in order to recover a stochastic process $p^{\varepsilon} (\omega,x,t)$ from solving a deterministic PDE for each $(y(s))_{0\leq s \leq t}$. This question is positively answered by the prominent works \cite{doss1977liens}, \cite{sussmann1978gap} which even prove that considering $C^1$ trajectories $y(t)$ is sufficient. As in the whole paper, this allows to consider $p^{\varepsilon} (x,t)$ as a deterministic function which depends on a given $C^1$ trajectory $(y(s))_{0\leq s \leq t}$. The function $p^{\varepsilon} (x,t)$ is thus the solution of a linear parabolic PDE, for which strong $C^2$ regularity can be shown using the classical theory.

\section*{Acknowledgements}
The authors would like to thank Kai Shi for his illustrations of the Skorohod dynamics. Philippe Moireau would like to sincerely thank Hasnaa Zidani for her guidance at the beginning of this work.

%=============================================================================================
% Bibliography 
%=============================================================================================
\bibliographystyle{plain} 
\bibliography{hjb-biblio.bib} 
%\printbibliography

\end{document}